\documentclass[11pt,english]{article}
\usepackage[margin = 2.5cm]{geometry}

\usepackage{amsmath}
\usepackage{amsthm}
\usepackage{amssymb}
\usepackage{mathtools}
\usepackage{graphicx}
\usepackage{xcolor}
\usepackage[hidelinks, bookmarks = false]{hyperref}
\usepackage{enumitem}
\usepackage{bbm}

\usepackage{caption}
\newtheorem{theorem}{Theorem}[section]
\newtheorem{corollary}[theorem]{Corollary}

\newtheorem{lemma}[theorem]{Lemma}

\newtheorem{definition}[theorem]{Definition}

\newtheorem{claim}[theorem]{Claim}

\def\b1K{\mbox{\boldmath $K$}_{-1}}
\def\bK{\mbox{\boldmath $K$}}

\newbox\noforkbox \newdimen\forklinewidth
\forklinewidth=0.3pt \setbox0\hbox{$\textstyle\smile$}
\setbox1\hbox to \wd0{\hfil\vrule width \forklinewidth depth-2pt
 height 10pt \hfil}
\wd1=0 cm \setbox\noforkbox\hbox{\lower 2pt\box1\lower
2pt\box0\relax}

\newcommand{\NN}{\mathbb{N}}

\def\sub'm{\prec_{\bK'}}
\def\grpf #1 #2{{\rm grp}_{#2}(#1)}
\def\fldf #1 #2{{\rm fld}_{#2}(#1)}
\def\dclf #1 #2{{\rm dcl}_{#2}(#1)}
\def\rclf #1 #2{{\rm rcl}_{#2}(#1)}
\def\aclf #1 #2{{\rm acl}_{#2}(#1)}
\def\acff #1 #2{{\rm acf}_{#2}(#1)}
\def\strf #1 #2{{\rm str}_{#2}(#1)}
\def\tclf #1 #2{{\rm acf}_{#2}(#1)}

\def\hbar{{\bf h}}

\date{\today}

\newcommand{\K}{\mathcal{K}}
\newcommand{\F}{\mathcal{F}}

\newcommand{\A}{\mathcal{A}}
\newcommand{\B}{\mathcal{B}}
\newcommand{\G}{\mathcal{G}}
\newcommand{\T}{\mathcal{H}}
\newcommand{\D}{\mathcal{D}}

\newcommand{\C}{\mathcal{C}}

\newcommand{\I}{\mathcal{I}}
\newcommand{\J}{\mathcal{J}}

\newcommand{\W}{\mathcal{W}}

\newcommand{\Z}{\mathcal{Z}}

\newcommand{\Hc}{\mathcal{H}}
\newcommand{\HH}{\mathcal{H}}
\newcommand{\Lc}{\mathcal{L}}

\newcommand{\ve}{\varepsilon}

\newcommand{\ex}{\mathrm{ex}}

\newcommand\MI{\mathrm{Int}}

\usepackage{mathtools}

\newcommand{\floor}[1]{\left\lfloor #1 \right\rfloor}

\newtheorem{prob}{Problem}

\usepackage{xpatch}
\xpatchcmd{\proof}{\itshape}{\normalfont\proofnamefont}{}{}
\newcommand{\proofnamefont}{}
\renewcommand{\proofnamefont}{\bfseries}
\newcommand{\forb}{{\mathrm{ forb}}}
\newcommand{\Forb}{{\mathrm{Forb}}}

\newcommand{\cross}{\sigma}

\newcommand{\ckl}{C^{(k)}_\ell}
\newcommand{\tkl}{T^{(k)}_\ell}
\newcommand{\gnpk}{G(n,p)^{(k)}}

\begin{document}

\title{On the number of $\HH$-free hypergraphs}

\author{Tao Jiang \footnote{Dept. of Mathematics, Miami University, Oxford, OH 45056, USA, {\tt jiangt@miamioh.edu}. Research supported
by the National Science Foundation grant DMS-1855542.  }\and Sean Longbrake \footnote{Dept. of Mathematics, Emory University,  Atlanta, GA 30322, USA {\tt sean.longbrake@emory.edu}} }

\maketitle

\begin{abstract}
Two central problems in extremal combinatorics are concerned with estimating the number $\ex(n,\HH)$, the size of the largest $\HH$-free hypergraph on $n$ vertices, and the number ${\rm forb}(n,\HH)$ of $\HH$-free hypergraph on $n$ vertices. It is well known that $\forb(n,\HH)=2^{(1+o(1))\ex(n,\HH)}$ for $k$-uniform hypergraphs that are not $k$-partite. In a recent breakthrough, Ferber, McKinley, and Samotij proved that for many $k$-partite (or {\it degenerate}) hypergraphs $\Hc$, $\forb(n, \Hc) = 2^{O(\ex(n, \Hc))}$. However, there are few known instances of degenerate hypergraphs $\HH$ for which $\forb(n,\HH)=2^{(1+o(1))\ex(n,\HH)}$ holds.


    In this paper, we show that $\forb(n,\HH)=2^{(1+o(1))\ex(n,\HH)}$ holds for a wide class of degenerate hypergraphs
known as $2$-contractible hypertrees. This is the first known infinite family of degenerate hypergraphs $\HH$ for which
$\forb(n,\HH)=2^{(1+o(1))\ex(n,\HH)}$ holds.
As a corollary of our main results, we obtain a sharp estimate of
$\forb(n,C^{(k)}_\ell)=2^{(\floor{\frac{\ell-1}{2}}+o(1))\binom{n}{k-1}}$ for the $k$-uniform linear $\ell$-cycle,
for all pairs $k\geq 5, \ell\geq 3$, thus settling a question of 
Balogh, Narayanan, and Skokan  affirmatively for all $k\geq 5, \ell\geq 3$. Our methods also lead to some sharp results on the related random Tur\'an problem.\\

As a key ingredient of our proofs,  we develop a novel supersaturation variant of the delta systems method for set systems, which may be of independent interest.

\end{abstract}

\section{Introduction}

\subsection{History}
In extremal combinatorics, the problem of enumerating the number of discrete structures that avoid given substructures
has a very rich history. One of the most natural questions one may ask is as follows: given a fixed graph $H$, how many $n$-vertex (labelled) graphs are there that contain no copy of $H$? Formally, given a fixed graph $H$, we say that a graph $G$ is {\it $H$-free} if it does not contain $H$ as 
a subgraph. For each natural number $n$, we let $\Forb(n,H)$ denote the family of all labeled $H$-free graphs on
the vertex set $[n]:=\{1,\dots, n\}$ and let $\forb(n,H)=|\Forb(n,H)|$. The problem is to determine or estimate $\forb(n,H)$.
This function is closely related to another classic function studied in extremal graph theory, namely
the {\it extremal number} $\ex(n,H)$, defined as the maximum number of edges in an $n$-vertex $H$-free graph. Indeed,
if we take a maximum $n$-vertex $H$-free graph $G$ and take all the subgraphs of it we get $2^{\ex(n,H)}$ many $H$-free graphs. 
On the other hand, for each $0\leq i \leq \ex(n,H)$ there are at most $\binom{\binom{n}{2}}{i}$ many $n$-vertex $H$-free graphs
with $i$ edges. Hence, we trivially have
\begin{equation}\label{graph-trivial}
2^{\ex(n,H)}\leq \forb(n,H)\leq \sum_{i\leq \ex(n,H)} \binom{\binom{n}{2}}{i}=n^{O(\ex(n,H))}.
\end{equation}
For non-bipartite graphs $H$, the upper bound in \eqref{graph-trivial} was significantly sharpened by Erd\H{o}s, Frankl, and R\"odl \cite{EFR}, extending the earlier seminal work of Erd\H{o}s, Kleitman, and Rothschild \cite{EKR-free} for complete graphs, showing that
$\forb(n,H)\leq 2^{\ex(n,H)+o(n^k)}$. Therefore, for any non-bipartite $H$,
\begin{equation} \label{graph-tight}
\forb(n,H)=2^{(1+o(1))\ex(n, H)}. 
\end{equation}
For bipartite graphs $H$, estimating $\forb(n,H)$ is much more difficult, even for the few bipartite $H$ whose extremal numbers $\ex(n,H)$ 
are relatively well-understood. The first breakthrough in this area was made by Kleitman and Winston \cite{KW}, who showed that the number of $C_4$-free graphs was $2^{O(\ex(n, C_4))}$. For complete bipartite graphs, this was extended by Balogh and Samotij first for symmetric \cite{BS-sym} and then asymmetric \cite{BS-asym} versions, showing that the number of $K_{s, t}$-free graphs is no more than $2^{O(n^{2 - 1/ s})}$, a near optimal result for $t$ sufficiently large compared to $s$. 
For even cycles, this was extended by Morris and Saxton \cite{MS} in a breakthrough work, showing that the number of $C_{2 \ell}$-free graphs is no more than $2^{O(n^{1 + \frac{1}{\ell}})}$. 
In a more recent breakthrough, Ferber, McKinley, and Samotij \cite{FMS} showed that for all bipartite graphs satisfying a 
mild condition (see their Theorem 5), $\forb(n,H)=2^{O(\ex(n, H))}$ holds. 
Until recently, it was believed that in fact $\forb(n,H)=2^{(1+o(1))\ex(n,H)}$ should hold for all bipartite graphs that contain a cycle
just as it does for non-bipartite $H$. But this was disproved by Morris and Saxton \cite{MS}, who showed that 
$\forb(n,C_6)\geq 2^{(1+c)\ex(n,C_6)}$ for some positive $c$. Thus, unlike for non-bipartite graphs, for bipartite $H$, the trivial lower bound of $2^{\ex(n,H)}$ is not always asymptotically tight.

One may consider the natural extension of the problem to the setting of uniform hypergraphs. For an integer $k\geq 2$,
a {\it $k$-uniform hypergraph} (or $k$-graph) is pair $(V,E)$ of finite sets, where the {\it edge set} $E$ is a family of
$k$-element subsets of the {\it vertex set} $V$. For a fixed $k$-graph $\HH$, one defines the extremal number $\ex(n,\HH)$ and $\forb(n,\HH)$
analogously as in the graph setting. As for graphs, for any $k$-graph $\HH$, we trivially have
\begin{equation}\label{hypergraph-trivial}
2^{\ex(n,\HH)}\leq \forb(n,\HH)\leq \sum_{i\leq \ex(n,\HH)} \binom{\binom{n}{k}}{i}=n^{O(\ex(n,\HH))}.
\end{equation}
A $k$-graph $\HH$ is {\it $k$-partite} (or {\it degenerate}) if its vertex can be partitioned into $k$ parts $X_1,\dots, X_k$
so that each edge contains exactly one vertex in each part. We call $(X_1,\dots, X_k)$ a {\it $k$-partition} of $\HH$.
It is easy to see that when $\HH$ is not $k$-partite, $\ex(n,\HH)=\Theta(n^k)$.
When $\HH$ is $k$-partite, it follows from a result of Erd\H{o}s \cite{Erdos-complete-partite} that $\ex(n,\HH)=O(n^{k-c})$, for some constant $c>0$.
Extending the work of Erd\H{o}s, Frankl and R\"odl \cite{EFR},
Nagle, R\"odl, and Schacht \cite{NRS} showed that for any fixed $k$-graph $\HH$,
$\mathrm{forb}(n, \HH) \leq 2^{\ex(n, \HH) + o(n^k)}$. Hence, for any $k$-graph $\HH$ that is not $k$-partite, we have
\begin{equation} \label{hypergraph-tight}
\forb(n,\HH)=2^{(1+o(1))\ex(n, \HH)}. 
\end{equation}

When $\HH$ is $k$-partite, however, the bound $\forb(n,\HH)\leq 2^{\ex(n,\HH)+o(n^k)}$ becomes too weak.
To improve this trivial upper bound for $k$-partite $k$-graphs, a natural approach is to first
study the problem for some prototypical examples of $k$-partite $k$-graphs.
It follows from the work of Balogh et al \cite{BDDLS} on the typical structures of $t$-intersecting families
that $\forb(n,\F_{k,t})=2^{(1+o(1))\ex(n,\F_{k,t})}$,  where $\F_{k,t}$ is the family of $k$-graphs each of which consists of
two edges sharing at least $t$ vertices in common. (Their actual estimate is in fact even sharper than stated,
see their Theorem 1.4). Mubayi and Wang \cite{MW} investigated $\forb(n,C^{(k)}_{\ell})$ for $k\geq 3$, where
$C^{(k)}_\ell$ is the so-called {\it $k$-uniform  linear cycle} of length $\ell$ which is the $k$-graph obtained from a graph
$\ell$-cycle by expanding each edge with $k-2$ degree $1$ vertices. It follows from the work of F\"uredi and Jiang \cite{FJ-cycles}
and of Kostochka, Mubayi and Verstra\"ete \cite{KMV-path-cycle} that for all $k,\ell\geq 3, \ex(n,C^{(k)}_\ell)\sim 
\floor{\frac{\ell-1}{2}} \binom{n}{k-1}$. Mubayi and Wang \cite{MW} 
showed that $\forb(n, C^{(3)}_\ell)=2^{O(n^{k-1})}$
for all even $\ell\geq 4$ and further conjectured that $\forb(n,C^{(k)}_\ell)=2^{O(n^{k-1})}$ for all $k,\ell\geq 3$.
Their conjecture was subsequently settled by Balogh, Narayanan and Skokan \cite{BNS}, who then posed the natural question
of whether $\forb(n, C^{(k)}_\ell)=2^{(1+o(1))\ex(n,C^{(k)}_\ell)}$ holds for all $k,\ell\geq 3$.
More recently, in the same paper mentioned earlier, Ferber, McKinley, and Samotij \cite{FMS} established the very general result that  for all $k$-partite
$k$-graphs $H$ that satisfy a mild condition (see their Theorem 9)
\begin{equation}
\forb(n,\HH)\leq 2^{O(\ex(n,\HH))},
\end{equation}
thus significantly sharpening the trivial upper bound in \eqref{hypergraph-trivial}, while also retrieving the  results
of Mubayi and Wang and of Balogh, Narayan and Skokan on $\ex(n,C^{(k)}_\ell)$.
In spite of this remarkable progress, it remains an intriguing question whether the bounds can be further sharpened,
and in particular, in view of \eqref{hypergraph-tight} and the question of Balogh, Narayanan and Skokan \cite{BNS}
whether there exists some nontrivial infinite family of $k$-partite $k$-graphs
for which $\forb(n,\HH)=2^{1+o(1))\ex(n,\HH)}$ holds. 

In this paper, we establish a large and the first known family of degenerate $k$-graphs,
for which $\forb(n,\HH)=2^{(1+o(1))\ex(n,\HH)}$ holds.  
As an immediate consequence of our main result, we settled
the question of Balogh, Narayanan and Skokan \cite{BNS} in the affirmative for all $k\geq 5, \ell \geq 3$.

\begin{theorem}  \label{cycles}
For all integers, $k\geq 5, \ell\geq 3$, we have
\[\forb(n,C^{(k)}_\ell)= 2^{(1+o(1)) \floor{\frac{\ell-1}{2}} \binom{n}{k-1}}.\]
\end{theorem}

To state our main results, we need a few definitions and history, which we detail next.

\subsection{Main Result}

First, we define hypertrees recursively as follows. A single edge $E$ is a hypertree.
In general, a hypergraph $\HH$ with at least two edges is a hypertree if there exists an edge $E$ such that
$\HH':=\HH\setminus E$ is a hypertree and there exists an edge $F$ in $\HH'$ 
such that $E\cap V(\HH')=E\cap F$; we call such an edge $E$ a {\it leaf edge} of $\T$ and call $F$ a {\it parent edge} of $E$
in $\HH'$. It follows from the definition above that if $\HH$ is a hypertree with $m\geq 2$ edges, there exists an ordering
of its edges as $E_1,\dots, E_m$ such that for each $i=2,\dots, m$, $\HH_i:=\{E_1,\dots, E_i\}$ is a tree and
$E_i$ is a leaf edge of $\HH_i$. We call such an edge-ordering a {\it tree-defining} ordering for $\HH$.
If $k$ is a positive integer, we will call a $k$-uniform hypertree a {\it $k$-tree}. 
As a simple example, a $k$-uniform matching is a $k$-tree.
It is easy to show by induction that every $k$-tree is $k$-partite.
Given a positive integer $t\leq k-1$, we say that a $k$-graph $\HH$ is {\it $t$-contractible} if each edge
of $\HH$ contains $t$ vertices of degree $1$ and a {\it $t$-contraction} of $\HH$ is the $(k-t)$-uniform muliti-hypergraph
obtained by deleting $t$ degree $1$ vertices from each edge of $\HH$. 
The notion of $t$-contractability is more general than that
of the {\it $k$-uniform expansion} of a graph. The latter has been quite extensively studied, as detailed in 
the survey by  Mubayi and Verstra\"ete \cite{MV-survey}. 
Specifically, the {\it $k$-expansion} of a graph $G$
is the  $k$-graph $G^{(k)}$ obtained from $G$ by expanding each edge into a $k$-set by adding $k-2$ degree $1$ vertices.
Thus, for instance, the $k$-uniform linear cycle $C^{(k)}_\ell$ is the $k$-expansion of a graph $\ell$-cycle.
Hence, the $k$-expansion of a graph $G$ is a $(k-2)$-contractible $k$-graph whose $(k-2)$-contraction is a simple
hypergraph with no repeated edges, whereas for a general $(k-2)$-contractible $k$-graph $\HH$, its $(k-2)$-contraction
is allowed to be a multigraph.

 Given a hypergraph $\HH$, a set $S\subseteq V(\HH)$ is called
a {\it cross-cut} of $\HH$ if each edge of $\HH$ intersects $S$ in exactly one vertex. If $\HH$ has a cross-cut then
we denote the minimum size of a cross-cut of it by $\sigma(\HH)$, and call it the {\it cross-cut number} of $\HH$.
Note that every $k$-partite $k$-graph $\HH$ has a cross-cut, for instance by taking any part in a $k$-partition of $\HH$.
Generalizing a long line of work \cite{Furedi-trees, FJ-cycles, FJS, KMV-path-cycle, KMV-trees} for $k$-graphs with $k\geq 5$, 
F\"uredi and Jiang \cite{FJ-trees} established the following sharp results on the extremal 
number of any subgraph of a $2$-contractible $k$-tree for $k\geq 5$.

\begin{theorem}[\cite{FJ-trees}] \label{FJ-trees}
Let $k\geq 4$ be an integer. For any $k$-graph $\HH$ that is a subgraph of a $2$-contractible $k$-tree, we have
\[\ex(n,\HH)=(\sigma(\HH)-1+o(1))\binom{n}{k-1}.\]
\end{theorem}
In particular, Theorem \ref{FJ-trees} implies that for $k \geq 5$ and $\ell \geq 3$, 
$\ex(n,   C_\ell^{(k)}) = \floor{\frac{\ell - 1}{2}} \binom{n}{k - 1}$. Indeed, for $k\geq 5$, it is easy to see
that $C_\ell^{(k)}$ is a subgraph of a $k$-tree and $\cross(C_\ell^{(k)})=\floor{\frac{\ell+1}{2}}$
(see \cite{FJ-cycles} for instance for details).
The authors of \cite{FJ-trees} also demonstrated existence of $1$-contractible hypertrees for which the conclusion
is no longer valid.
The main method used in \cite{FJ-trees} is the so-called {\it Delta system method}, which is 
a powerful tool for studying extremal problems on set systems.

In this paper, we develop a supersaturation variant of the Delta system method and use it in conjunction with the container method
to establish the following sharp enumeration result.

\begin{theorem} [Main Theorem] \label{main}
Let $k\geq 4$ be an integer. For every $2$-contractible $k$-tree $\HH$, we have
\[\forb(n,\HH)= 2^{(\cross(\HH) - 1 +o(1))\binom{n}{k - 1}}.\]
\end{theorem}

This gives the first known family of degenerate $k$-graphs $\HH$ for which $\forb(n,\HH)=2^{(1+o(1))\ex(n,\HH)}$ holds.

Furthermore, as a by-product of the supersaturation variant of the Delta-system method, we also get an optimal supersaturation result
for the family of $2$-contractible $k$-trees for $k\geq 4$, which may be of independent interest (see Theorem~\ref{supersat}).

As an immediate corollary of Theorem \ref{main}, we obtain Theorem \ref{cycles},
which answers the question of Balogh, Narayanan and Skokan \cite{BNS} affirmatively for all $k\geq 5,\ell\geq 3$.
Indeed, as mentioned earlier, $\cross(C^{(k)}_\ell)= \floor{\frac{\ell+1}{2}}$, and by considering $\ell$ even and $\ell$ odd cases 
separately it is not hard to construct a $k$-tree $\tkl$
with cross-cut number $\floor{\frac{\ell+1}{2}}$ that contains $\ckl$. Thus, by Theorem \ref{main},
$\forb(n, \ckl)\leq \forb(n, \tkl) \leq 2^{(\floor{\frac{\ell-1}{2}}+o(1))\binom{n}{k-1}}$, 
for all $k \geq 5$ and $\ell \geq 3$, from which Theorem \ref{cycles} follows.

\subsection{Applications to the random Tur\'an problem}
The methods used in establishing our main theorem also readily yield some sharp results on 
the so-called {\it random Tur\'an problem}. Let $k\geq 2$ be an integer. We let $G_{n, p}^{(k)}$ be the so-called
{\it Erd\H{o}s-R\'enyi random graph} formed by keeping each edge of $K_n^{(k)}$ uniformly at random with probability $p$. 
The random Tur\'an problem is to study the random variable $\ex(G_{n, p}^{(k)}, H)$ 
that counts the maximum number of edges in  an $H$-free subgraph of 
$G_{n,p}^{(k)}$. For a thorough introduction to the random Tur\'an problem, the reader is referred to the excellent survey by R\"odl and Schacht~\cite{RS}, though we will offer a brief summary here.

 The random Tur\'an problem for non-$k$-partite $k$-graphs  was essentially solved
 in breakthrough works by Conlon and Gowers \cite{CG}
and by Schacht \cite{Schacht}, who showed
that when $p\gg n^{-1/m_k(H)}$, almost surely $\ex(n,\gnpk) =p(\ex(n,H)+o(n^k))$, as $n\to \infty$,
where $m_k(H)=\max_{F\subseteq H, e(F)\geq 2} \frac{e(F)-1}{v(F)-k}$.
The theorem was then reproved using the container method by Balogh, Morris and Samotij \cite{BMS} and independently by Saxton and Thomason~\cite{ST}.

For $k$-partite $k$-graphs (which are also called {\it degenerate $k$-graphs})
much less is known. For degenerate graphs, early results due to Haxell, Kohayakawa, and {\L}uczak \cite{HKL} and Kohayakawa, Kreuter, and Steger \cite{KKS} essentially solved the problem for even graph cycles for small values of $p$. For dense ranges, Morris and Saxton \cite{MS} solved the problem for cycles and complete biparite graphs, and McKinley and Spiro  \cite{theta} extended their result for theta graphs. For dense $p$, some general upper bounds are obtained in \cite{JL}.  

Even less is known for degenerate hypergraphs. For $k$-uniform even linear cycles, Mubayi and Yepremyan \cite{MY} and independently Nie \cite{Nie-cycle} proved 
\begin{theorem}[\cite{MY, Nie-cycle}]\label{linear-cycles-old-bounds}
For every $\ell\geq 2$ and $k\geq 4$ with high probability, the following holds: 
$$ \ex\left(G_{n,p}^{(k)}, C_{2\ell}^{(k)}\right) = \begin{cases} 
\Theta(pn^{k-1}), & \mbox{if } { p \geq n^{-(k-2)+\frac{1}{2\ell-1}+o(1)}} \\
n^{1+\frac{1}{2\ell-1}+o(1)}, & \mbox{if } {n^{-(k-1)+\frac{1}{2\ell-1}+o(1)}\leq p\leq  n^{-(k-2)+\frac{1}{2\ell-1}+o(1)}} \\ 
(1-o(1))pn^k,  & \mbox{if } {n^{-k} \ll p\ll n^{-(k-1)+\frac{1}{2\ell-1}}} .
\end{cases}$$
\end{theorem}
For other classes of hypergraphs, near optimal results were obtained by Nie \cite{Nie-expand} for $k$-expansions of subgraphs of tight $p$-trees and of $K_p^{p-1}$, where $k\geq p\geq 3$, which includes both odd and even linear cycles. Nie and Spiro \cite{NS} also were able to get near optimal bounds for expansions of theta-graphs. Furthermore, for any $k$-uniform hypergraph $H$, which satisfies that for every $k$-uniform $G$ there are at least $e(G)^{e(H)}v(G)^{v(H) - k e(H)}$ many homomorphisms from $H \rightarrow G$, they proved that there is some $r_0 \geq k$ such that for every $r \geq r_0$, tight bounds on the random Tur\'an number of $r$-expansions of $H$ hold.

As an immediate byproduct of our main result, for dense $p$ and for $\ckl$ with $k \geq 5$, we are able to 
sharpen the bound $pn^{k-1+o(1)}$ to the correct bound  
$p(\floor{\frac{\ell-1}{2}}+o(1))\binom{n}{k-1}=p(1+o(1)) \ex(n,\ckl)$,
albeit for a slighly more restricted range of $p$.
This comes as an immediate corollary to the following more general theorem. 
A hypergraph $\HH$ is $\ell$-overlapping if for any two edges $E,F$ in $\HH$, $|E\cap F|\leq \ell$.

\begin{theorem}\label{random-tree} Let $k\geq 4$ be an integer. Let $\HH$ be a $2$-contractible $\ell$-overlapping $k$-tree. When $p \gg \frac{\log (n)^2}{n^{ k -\ell - 1}}$, with high probability 
\[\ex(\gnpk, \HH)= (\cross(\HH) - 1 +o(1))p \binom{n}{k - 1} = p(1+o(1))\ex(n,\HH).\]
\end{theorem}

As mentioned earlier, for all $k\geq 5, \ell \geq 3$, there exists a $2$-contractible
$k$-tree $\tkl$ containing $\ckl$ with $\sigma(\tkl)=\sigma(\ckl)=\floor{\frac{\ell+1}{2}}$. Hence,
Theorem \ref{random-tree} immediately implies

\begin{corollary} For all $k\geq 5, \ell\geq 3$ and $p \gg \log (n)^2 n^{ -(k - 3)}$, with high probability,
\[\ex(\gnpk, \ckl) = p\left(\floor{\frac{\ell-1}{2}}+o(1)\right )\binom{n}{k-1} 
=p(1+o(1))\ex(n,\ckl).\]
\end{corollary}

\subsection{Overiew of methodology and organization of the paper}

At the heart of our work is the so-called optimal balanced supersaturation at the Tur\'an thresold. Several long-standing conjectures
of Erd\H{o}s and Simonovits \cite{ES-cube} addressed the question of how many copies of a graph $H$  we can guarantee
in a dense enough host graph $G$. Loosely speaking, the conjectures say that the number of copies of $H$
we expect in $G$ should be at least on the same order of magnitude as the number of copies of $H$ we expect
in a random graph with the same edge-density as $G$. These are referred to as supersaturation conjectures.
The strongest of the Erd\H{o}s-Simonovits supersaturation conjectures says that the conjectured bound 
on the number of copies of $H$ should already hold as soon as an $n$-vertex graph $G$ has just barely asymptotically a bit
more edges than what is enough to guarantee a single copy of $H$, namely when $e(G)\geq (1+\varepsilon) \ex(n,H)$
for any small real $\varepsilon>0$. We will refer to this phenomenon as {\it optimal supersaturation at the Tur\'an threshold}.
Establishing optimal supersaturation at the Tur\'an threshold turns out to be a very difficult task for degenerate (i.e bipartite) 
graphs, with the problem being unsolved except for very few bipartite graphs.

In the last decade or so, the development of the container method has brought
enhanced importance to the supersaturation problem. Specifically, the container method allows one
to obtain tight enumeration results on $\forb(n,H)$ once one is able to obtain supersaturation of $H$ with the
additional feature that the copies of $H$ found are evenly distributed in a sense. While this paved the way
for several breakthroughs mentioned earlier, naturally developing optimal supersaturation at the Tur\'an threshold
with the added balanced feature is an even more difficult task than the one without the additional balanced requirement, as witnessed by the fact this has not been done even
for the $4$-cycle, whose extremal number is very well-understood \cite{Furedi-C4} and for whom supersaturation at the Tur\'an threshold has been achieved \cite{ES-sat}.

Our main results crucially build on the optimal balanced supersaturation at the Tur\'an threshold for $2$-contractible hypertrees.
Given the difficulty with the optimal balanced supersaturation at the Tur\'an threshold for graphs, it does come
as a surprise that one is able to establish it for a large family of degenerate hypergraphs. This is in a strong sense 
attributed to the power of the method we use, known as the Delta system method for set-systems. However, while the Delta system method
has been successfully used on Tur\'an type problem for hypergraphs, it has not been tailored for supersaturation problems before.
In that regard, the most important innovative aspect of our work is the development of a supersaturation variant of the Delta system method
and applying it successfully with the container method to get tight enumeration results. We believe that this variant of the Delta system 
method will find future applications.

We organize the rest of the paper as follows. 
In section \ref{sec:notation}, we give some notation. In Section \ref{sec:delta},
we develop a supersaturation variant of the Delta system method and develop
a structural dichotomy for all $k$-graphs with $\Theta(n^{k-1})$ edges, both of which may be of independent interest.
In Section \ref{sec:supersaturation}, we develop optimal supersaturation at the Tur\'an threshold for $2$-contractible hypertrees.
In Section \ref{sec:balanced-supersaturation}, we develop optimal balanced supersaturation at the Tur\'an thresold for $2$-contractible hypertrees. In Section \ref{sec:main-proofs}, we prove our main theorem, Theorem \ref{main}, as well as Theorem \ref{random-tree}.

\section{Notation} \label{sec:notation}

Let $\F$ be a hypergraph on $V=V(\F)$. For each integer $i\geq 0$,  we define the {\it $i$-shadow} of $\F$ to be
\[\partial_i(\F):=\{D: |D|=i, \exists F\in \F, D\subseteq F\}.\]

The Lov\'asz'~\cite{L79} version of the Kruskal-Katona theorem states that if
 $\F$ is a $k$-graph of size $|\F|=\binom{x}{k}$, where $x\geq k-1$ is a real number,
 then for all $i$ with $1\leq i \leq k-1$ one has
\begin{equation}\label{lovasz-bound}
  |\partial_i(\F)|\geq \binom{x}{i}.
  \end{equation}
  
Given $D\subseteq V(\F)$, we define the {\it link of} $D$ in $\F$ to be
\[\Lc_\F(D)=\{F\setminus D: F\in \F, D\subseteq F\}.\]
Note that we allow $\emptyset$ to be a member of $\Lc_\F(D)$. We define the {\it degree of } $D$ in $\F$ to
be $d_\F(D):=|\Lc_\F(D)|$.

Given a $k$-graph $\F$ where $k\geq 2$ is an integer and integer $i$ with $1\leq i\leq k-1$, let
\[\Delta_i(\F)= \max \{d_\F(D): D\in \partial_i(\F)\} \mbox { and } \delta_i(\F):=\min\{d_\F(D): D\in \partial_i(\F)\}.\]
We call $\delta_i(\F)$ the {\it proper minimum $i$-degree} of $\F$. By definition, every $i$-set in $\F$
either has degree $0$ or has degree at least $\delta_i(\F)$. Let $\F$ be a $k$-graph and $S\subseteq V(\F)$,
we let $\F-S:=\{F\setminus S: F\in \F\}$



\section{A variant of the Delta system method and a structural dichotomy for $k$ graphs of size $\Theta(n^{k - 1})$}
\label{sec:delta}

 The {\it delta system method}, originated by Deza, Erd\H{o}s and Frankl \cite{DEF}
and others, is a powerful tool for solving extremal set problems. A particularly versatile tool within the method, which one may call 
{\it the intersection semilattice lemma} was developed by F\"uredi \cite{Furedi-1983} (Theorem 1'). The delta system method, particularly aided
by the semilattice lemma, has been 
very successfully used to obtain a series of sharp results on extremal set problems (see for instance \cite{FF-1985,
FF-exact, Furedi-1983, Furedi-trees, FJ-cycles, FJS}). However, despite its effectiveness in determining the threshold
on the size of a hosting hypergraph beyond which a certain subgraph occurs, it does not readily allow us
to effectively count the number of such subgraphs (known as the supersaturation problem).
In this section, we develop a variant of F\"uredi's intersection semilattice lemma, Theorem~\ref{homogeneous} below,
to also address hypergraph supersaturation.

Let $k\geq 2$ be an integer. Let $\F$ be a $k$-partite $k$-graph $\F$ with a fixed $k$-partition $(X_1,\dots, X_k)$.
For each $J\subseteq [k]$ and $F\in \F$,  we define the {\it $J$-projection} of $F$, denoted by $F_J$ to be

\[F_J:= F\cap (\bigcup_{i\in J} X_i).\]
Conversely, for any $D\in \bigcup_{i=1}^k \partial_i(\F)$, we define the {\it pattern} of $D$, denoted by $\pi(D)$, to be
\[\pi(D):=\{i\in [k]: D\cap X_i\neq \emptyset\}.\]
Note that since $\F$ is $k$-partite, $|\pi(D)|=|D|$ for each $D\in \bigcup_{i=1}^k \partial_i(\F)$.

Given a positive integer $s\geq 2$, a nonempty hypergraph $\Hc$ is called {\it $s$-diverse} if 
\[\forall v\in V(\Hc), d_\Hc(v)< (1/s)|\Hc|.\]
Note that implicitly an $s$-diverse hypergraph necessarily contains more than $s$ edges.

Let
\[\MI(\F)=\{\pi(E\cap F): E,F\in \F, E\neq F\}.\]


\begin{theorem}[Super-homogeneous Subfamily Lemma] \label{homogeneous} 
Let $s, k\geq 2$ be integers where $s \geq 2k$. Let $c(k,s) =  \frac{k!}{k^k (2s(1 + 2^{k}))^{2^k}}$. 
Let $\F$ be a $k$-graph on $[n]$.
Then there exists a $k$-partite subgraph $\F' \subseteq \F$ with some $k$-partition $(X_1,\dots, X_k)$ such that the following holds: 
\begin{enumerate}
    \item $|\F'|\geq c(k,s) |\F|$.
\item For every $F\in \F'$ and $J\in \MI(\F')$, $\Lc_{\F'}(F_J)$ is $s$-diverse. 
\item For every $F\in \F'$ and $J\in \MI(\F')$, $d_{\F'}(F_J)\geq \max\{s,  \frac{1}{2k}\frac{|\F'|}{n^{|J|} }\}$.
\end{enumerate}
We call such a $k$-graph $\F'$ $s$-super-homogeneous. 
\end{theorem}
\begin{proof}
By a well-known result of Erd\H{o}s and Kleitman \cite{EK}, $\F$ contains a $k$-partite subgraph $\F_0$
with $|\F_0|\geq \frac{k!}{k^k}|\F|$.  Let $(X_1, X_2, \dots X_k)$ be a fixed $k$-partition of $\F_0$.
If $\F_0$ satisfies conditions 2 and 3, then the theorem holds with $\F'=\F_0$. Otherwise, let 
$\G_0=\F_0$. We perform a so-called {\it filtering process} on $\G_0$ as follows.
Let $\W_0=\emptyset$.
For each $J\in \MI(\G_0)$, let $\A_J=\emptyset$. We iteratively modify $\G_0$, $\W_0$  and the $\A_J$'s for $J\in \MI(\G_0)$ as follows.
 Whenever there is an edge $F\in \G_0$ and a $J\in \MI(\G_0)$ such that
$\Lc_{\G_0}(F_J)$ is nonempty and not $s$-diverse, we remove all the edges of $\G_0$ containing $F_J$ and
add them to $\A_J$. Whenever there is an edge $F\in \G_0$ and a $J\in \MI(G_0)$ such that $\Lc_{\G_0}(F_j)$ is $s$-diverse
but $d_{\G_0}(F_J)<\frac{1}{2k}\frac{|\F_0|}{n^{|J|}}$, we remove all the edges of $\G_0$ containing $F_J$ and add them to $\W_0$.
Let $\G^*_0$ denote the final $\G_0$ at the end of the filitering process. By definition, if $\G^*_0$ is nonempty, then $\G^*_0$ satisfies conditions 2
and 3.

Note that  $|\W_0|\leq  |\F_0|/2$. This is because for any fixed $j=1,\dots, k-1$, there are at most $\binom{n}{j}$
different $F_J$'s. When all edges containing some $F_J$ are moved to $\W_0$, by definition, fewer than $\frac{1}{2k}\frac{|\F_0|}{n^{|J|}}$ edges 
are moved.  Hence $|\F_0\setminus \W_0|\geq |\F_0|/2$.
If $|\G_0^*|\geq \frac{1}{1+2^k} |\F_0\setminus \W_0|$, we let $\F'=\G^*_0$.
Otherwise, by the pigeonhole principle, there exists some $J\in \MI(\G_0)$ such that $|\A_J|\geq \frac{1}{1+2^k} |\F_0\setminus \W_0|$. 
Note that edges were added to $\A_J$ in batches, with each  batch consisting of edges $F$ with the same $J$-projection and
different batches have different $J$-projections. Consider any batch $\B$  added to $\A_J$. Let $D$ denote the common $J$-projection
of the edges in $\B$. By definition, $\Lc_\B(D)=\Lc_{\G_0}(F_J)$ is nonempty and not $s$-diverse at the moment $\B$ was added to $\A_J$.
By definition, there exists a vertex $v$ in $\Lc_\B(D)$ that lies in at least $(1/s)|\Lc_\B(D)|$ of the edges. 
Let $\B'$ denote the subset of edges in $\B$ that also contain $v$. We now remove $\B$ from $\A_J$ and replace it with $\B'$.
We do this for each batch of edges that were added to $\A_J$, and denote the resulting subgraph of $\A_J$ by $\A'_J$.
Then $|\A'_J|\geq (1/s)|\A_J|$. Furthermore, it is easy to see that $J\notin \MI(\A'_J)$.
We let $\F_1=\A'_J$. Then 
\[|\F_1|\geq \frac{1}{s(1+2^k)}|\F_0\setminus \W_0|\geq \frac{1}{2s(1+2^k)}|\F_0|  \quad \mbox{ and } \quad |\MI(\F_1)|\leq |\MI(\F_0)|-1.\]

Now, let $\G_1=\F_1$, let $\W_1=\emptyset$ and set $\A_J=\emptyset$ for all $J\in \MI(\G_1)$. We 
then perform the same filtering on $\G_1$ to iteratively modify $\G_1$, $\W_1$ and the $\A_J$'s for $J\in \MI(\G_1)$. Let $\G^*_1$
denote the final $\G_1$ at the end of the filtering process. By definition, $\G^*_1$ satisfies conditions 2 and 3.
 If $|\G_1^*|\geq \frac{1}{1+2^k}|\F_1\setminus \W_1|$, we let $\F'=\G_1^*$.
Otherwise, as before, there exists some $J\in \MI(\G_1)$ and a subgraph $\A'_J\subseteq \A_J$ with $|\A'_J|\geq \frac{1}{2s(1+2^k)} |\F_1|$
such that $|\MI(\A'_J)|\leq  |\MI(\G_1)|-1$. We let $\F_2=\A'_J$.

We continue like this, obtaining a sequence $\F_0, \F_1,\dots$. Since $|\MI(\F_i)|$ strictly decreases with $i$, the sequence must end with $\F_m$
for some $m\leq 2^k-1$. Since $\F_{m+1}$ is undefined, this must mean that 
\[|\G_m^*|\geq \frac{1}{1+2^k}|\F_m\setminus \W_m|\geq \frac{1}{2(1+2^k)}|\F_m|\geq \frac{1}{[2s(1+2^k)]^{2^k}}|\F_0|\geq c(k,s)|\F|,\]
where $c(k,s)=\frac{k!}{k^k}\cdot  \frac{1}{[2s(1+2^k)]^{2^k}}$. Let $\F'=\G^*_m$. Then $|\F'|\geq c(k,s)|\F|$ and $\F'$ also satisfies conditions
2 and 3, by the definition of $\G^*_m$.
\end{proof}

Throughout the rest of the paper, whenever we consider an $s$-super-homogeneous $k$-graph $\F$, we always implicitly
fix a $k$-partition associated with $\MI(\F)$. 

Next, we collect some useful facts about $s$-super-homogeneous families, for which we need the following definition. 

\begin{definition}
{\rm 
Let $k$ be a positive integer. Given a family $\J$ of proper subsets of $[k]$,  let 
\[r(\J):=\min\{|D|: D\subseteq [k], \not\exists J\in \J \text{ such that } D\subseteq J\}.\] We call $r(\J)$ the {\it rank} of $\J$.}
\end{definition}
\begin{lemma} \label{semilattice}
Let $s , k \geq 2$ be integers with $s \geq 2k$. Let $\F$ be an $s$-super-homogeneous $k$-partite $k$-graph on $[n]$ with some fixed $k$-partition
$(X_1,\dots, X_k)$. Then, the following hold.
\begin{enumerate}
\item For each $J\subsetneq [k]$ where $J\notin \MI(\F)$ and each $F\in \F$, there is no $F'\in \F$ satisfying $F\cap F'=F_J$.
\item For all $ J,J'\in \MI(\F)$, $J\cap J'\in \MI(\F)$, that is,
$\MI(\F)$ is closed under intersection.
\item If $\MI(\F)$ has rank $m$, then $|\F|\leq \binom{n}{m}$.
\end{enumerate}
\end{lemma}
\begin{proof}
Statement 1 follows from definition of $\MI(\F)$. 
For statement 2, let $J,J'\in \MI(\F)$ with $J\neq J'$. Let $F\in \F$. By our assumption $\Lc_\F(F_J)$ is $s$-diverse.
So the vertices in $F\setminus F_J$ block fewer than $k(1/s)|\Lc_\F(F_J)| \leq |\Lc_\F(F_J)|$ of the edges in $\Lc_\F(F_J)$. So there exists $F'\in \F$ with
$F\cap F'=F_J$. By a similar reasoning, since $2k(1/s)|\Lc_{\F}(F'_{J'})| \leq |\Lc_{\F}(F'_{J'})|$ there exists $F''\in \F$ containing $F'_{J'}$ that avoids vertices in $(F\setminus F')\cup (F'\setminus F'_{J'})$. Now, $\pi(F\cap F'')=J\cap J'$. Hence, $J\cap J'\in \MI(\F)$.

For statement 3, suppose $\MI(\F)$ has rank $m$. Then $\exists D\subseteq[k]$ with $|D|=m$ such that $D$ is not contained in any member of
$\MI(\F)$. Consider any $F,F'\in \F$, where $F\neq G$. If $F[D]=F'[D]$, then $\pi(F\cap F')$ is a member of $\MI(\F)$ that contains $D$,
a contradiction. So, the $D$-projections of members of $\F$ are all distinct.
This implies that $|\F|\leq \binom{n}{m}$.
\end{proof}

The following structural lemma strengthens Lemma 7.1 of \cite{FF-exact} (see also Lemma 4.2 of \cite{FJ-trees})
and leads to a structural dichotomy theorem (Theorem \ref{centralization}) that is important for our main arguments.
\begin{lemma}\label{lem:types}
Let $k\geq 3$ be an integer.
Let $\J\subseteq 2^{[k]}$ be a family of proper subsets of $[k]$ that is closed under intersection.
 Suppose $\J$ has rank at least $k-1$. Then
one of the following must hold.
\begin{enumerate}
\item There exists $B\subseteq [k]$, with $|B|=k-2$, such that $2^B\subseteq \J$. 
\item There exists a unique $i\in [k]$ such that $\forall D\subsetneq [k]$ with $i\in D$ we have
$D\in \J$ and for every $D\subseteq [k]\setminus \{i\}$ with $|D|\geq k-2$ we have $D\notin \J$.
We call $i$ {\it the central index} for $\J$. 
\end{enumerate}
If condition 1 holds for $\J$, we say that $\J$ is of type 1. If condition 2 holds for $\J$,
we say that $\J$ is of type 2.
\end{lemma}
\begin{proof} If $r(\J)=k$ then $[k]\setminus \{j\}\in \J$ for each $j\in [k]$. Since $\J$ is closed under intersection,
we see that $S\in \J$ for each proper subset $S$ of $[k]$. Hence statement 1 clearly holds. Next, suppose $r(\J)=k-1$. 
Then some $(k-1)$-subset  of $[k]$ is not in $\J$. Without
loss of generality, suppose $[k]\setminus \{j\}\notin \J$ for $j=1,\dots, t$ and $[k]\setminus \{j\}\in \J$ for $j=t+1,\dots, k$, 
for some  $1\leq t\leq k$. 

First, suppose $t=1$. Then $[k]\setminus\{1\}\notin \J$ and $[k]\setminus \{2\},\dots, [k]\setminus \{k\}\in \J$.
Since $\J$ is closed under intersection, every proper subset of $[k]$ that contains $1$ is in $\J$. 
We already have $[k]\setminus \{1\}\notin \J$.
Suppose there is a $(k-2)$-subset $B$ of $[k]\setminus\{1\}$ that is in $\J$. 
Let $S$ be any subset of $B$. By earlier discussion $S\cup\{1\}\in \J$.
Since $\J$ is closed under intersection, we have $S=(S\cup \{1\})\cap B\in \J$.
Hence, $2^B\subseteq \J$. So statement 1 holds in this case. Hence, we may assume
that no $(k-2)$-subset of $[k]\setminus\{1\}$ is in $\J$. Then statement 2 holds for $i=1$.
It is easy to see that if an $i$ satisfies the requirements, it can only be $1$.

Next, suppose $t\geq 2$. Observe that for any $1\leq i<j\leq t$, we must have $[k]\setminus \{i,j\}\in \J$,
as otherwise $[k]\setminus \{i,j\}$ is not contained in any member of $\J$,
contradicting $r(\J)\geq k-1$. Also, by our assumption for every $t+1\leq j\leq k$
we have $[k]\setminus \{j\}\in \J$. Since $\J$ is closed under intersection, we see that every subset of $[k]\setminus \{1,2\}$
is in $\J$. So statement 1 holds.
\end{proof}

Now, we describe our structural theorem that gives a dichotomy of $n$-vertex $k$-graphs with $\Theta(n^{k-1})$ edges.
We believe that it is also of independent interest.

\begin{theorem}\label{centralization}
Let $a, \ve$ be fixed positive reals with $\ve < \min\{1,  a^{ k - 2}\}$. Let $k,s$ be fixed positive integers with $s \geq 2k$. Let $n$ be sufficiently large
in terms of $a,\ve, k,s$. 
Let $\F$ be a $k$-graph on $[n]$ with $|\F|=a\binom{n}{k-1}$. Then one of the following two holds:
\begin{enumerate}
\item $\F$ contains an $s$-super-homogeneous $k$-partite subgraph $\F'$ such that $|\F'|\geq \frac{1}{2}c(k,s)\ve |\F|$ and
$\MI(\F')$ is of type 1, where $c(k,s)$ is the constant as in Theorem \ref{homogeneous}.
\item
There exist a $W \subseteq [n]$ and a subgraph $\F''\subseteq \F$ satisfying the following: 
\begin{enumerate}
\item $|W|\leq \left(\frac{10}{c(k,s)\ve^2}\right)^{k-1}$. 
\item  $|\F''|\geq (1-\ve) |\F|$.
\item  Every $F \in \F''$ satisfies $|W \cap F| = 1$. 
\end{enumerate}
\end{enumerate}
\end{theorem}
\begin{proof}
First we apply Lemma~\ref{homogeneous} to $\F$ to get an $s$-super-homogeneous subgraph $\F_1$
of size at least $c(k,s)|\F|$.
If $|\F\setminus \F_1|\leq \frac{1}{2} \ve |\F|$, then we stop. Otherwise, we apply the lemma to $\F\setminus \F_1$
to find an $s$-super-homogeneous subgraph  $\F_2$ of size at least $c(k,s)|\F\setminus \F_1|$.
In general, as long as $|\F\setminus (\F_1\cup\cdots\cup \F_i)|>\frac{1}{2}\ve|\F|$, we let $\F_{i+1}$
be an $s$-super-homogeneous subgraph of $\F\setminus (\F_1\cup \cdots\cup \F_i)$ of size
at least $c(k,s)|\F\setminus (\F_1\cup \cdots \cup \F_i)|$, which exists by Lemma \ref{homogeneous}. Suppose the sequence of subgraphs
we constructed are $\F_1,\dots, \F_m$. 

By our algorithm, for each $i\in [m]$, $|\F_i|\geq \frac{1}{2}c(k,s) \ve |\F|$. 
Note that this implies that $m\leq \frac{2}{c(k,s)\ve}$.
For each $i\in [m]$, by Lemma \ref{semilattice} part 3, $r(\MI(\F_i))\geq k-1$
and by Lemma \ref{semilattice} part 2 and Lemma~\ref{lem:types}, $\MI(\F_i)$ is either of type 1 or is of type 2.
If for some $i\in [m]$, $\MI(\F_i)$ is of type 1,
then we let $\F'=\F_i$ for such an $i$ and Statement 1 holds. 
Hence, we may assume that for each $i\in [m]$, $\MI(\F_i)$ is of type 2.
Let $\F^*=\bigcup_{j=1}^m \F_j$. Then 
\[|\F^*|\geq (1-\frac{\ve}{2})|\F|.\]

Set $h=\floor{\left(\frac{5m}{\ve}\right)^{k-1}}$. Since $m\leq \frac{2}{c(k,s)\ve}$, we have
\begin{equation} 
h\leq \left(\frac{10}{c(k,s)\ve^2}\right)^{k-1}.
\end{equation}

For each $j\in [m]$,
let $i_j\in [k]$ denote the central index for $\MI(\F_j)$ and for each $F\in \F_j$, let $c(F) = F_{\{i_j\}}$. 
We call $c(F)$ {\it the central vertex} of $F$.
We partition $\F^*$ according to $c(F)$. For each $i \in [n]$, let 
\[\A_i = \{ F \in \F^* : c(F) = i\} \text{ and } \A_i' = \{F \setminus \{i\}: F \in \A_i\}.\]
Clearly, for each $i\in [n]$ we have $|\A_i|=|\A'_i|$.
For each $j \in [m]$, let 
\[\F_j' = \{ F \setminus \{c(F)\} : F \in \F_j\}=\{F_{[k]\setminus \{i_j\}}: F\in \F_j\}.\]
Then, in particular: 
\[\bigcup_{i = 1}^n \A_i' = \bigcup_{j = 1}^m \F_j'. \]

The next claim is crucial to our argument.
\begin{claim} \label{disjoint}
For each $j\in [m]$, we have
\[|\partial_{k-2}(\F_j')|=\sum_{i = 1}^n  |\partial_{k - 2} (\A_i' \cap \F_j')|.\]
\end{claim}
\begin{proof}[Proof of Claim \ref{disjoint}]
Let $(X_1,\dots, X_k)$ denote the $k$-partition associated with $\MI(\F_j)$.
Recall that $i_j\in [k]$ denotes the central index for $\MI(\F_j)$. Without loss of generality,
we may assume $i_j=1$. To prove the claim,
it suffices to show that for any $a,b\in X_1$ where $a\neq b$, we have 
$\partial_{k-2}(\A'_a\cap \F'_j)\cap \partial_{k-2}(\A'_b\cap \F'_j)=\emptyset$. Suppose for contradiction
that there exists $S\in \partial_{k-2}(\A'_a\cap \F'_j)\cap \partial_{k-2}(\A'_b\cap \F'_j)$. Then
there exist $F_a\in \A_a\cap \F_j, F_b\in \A_b\cap \F_j$ such that $F_a\cap F_b\supseteq S$. 
By definition, $\pi(F_a\cap F_b)\in \MI(\F_j)$.
But $\pi(F_a\cap F_b)$ is a subset of $[k]\setminus \{1\}$ with size at least $k-2$. This
contradicts $\MI(\F_j)$ being of type 2
(see Lemma \ref{lem:types} (2)).  The claim immediately follows.
\end{proof}
Note that trivially $|\partial_{k - 2} (\F_j')| \leq \binom{n}{k - 2}$. 
Now, we have 
\begin{equation} \label{shadow-bound}
    \sum_{i = 1}^n |\partial_{k - 2}(\A_i')| \leq  \sum_{i = 1}^n\sum_{j = 1}^m |\partial_{k - 2}(\A_i' \cap \F_j')|\\
    = \sum_{j = 1}^m  \sum_{i = 1}^n|\partial_{k - 2}(\A_i' \cap \F_j')|
    = \sum_{j =1}^m |\partial_{k - 2} (\F_j')|\\
    \leq m \binom{n}{k - 2}
\end{equation}

For each $i \in [n]$, let $x_i$ be the real such that $|\partial_{k - 2}(\A_i')| = \binom{x_i}{k - 2}$, where without loss of generality, $x_1 \geq x_2 \dots \geq x_n$. By \eqref{lovasz-bound}, for each $i\in [n]$ we have
$|\A_i'| \leq \binom{x_i}{k - 1}$.

By \eqref{shadow-bound}, $\sum_{i=1}^n\binom{x_i}{k-2}\leq m\binom{n}{k-2}$. Thus,
$\binom{x_{h+1}}{k - 2} \leq \frac{m}{h+1} \binom{n}{ k - 2}$, which yields
\begin{equation} \label{xh-bound}
x_{h+1} - k + 3 \leq \left( \frac{m}{h+1} \right)^{ 1 / (k - 2)} n.
\end{equation}
Thus, we have
\begin{align*}
\sum_{i > h} |\A_i| =\sum_{i>h}|\A'_i|&\leq \sum_{i = h + 1}^n \binom{x_i}{k - 1} \\
&=\sum_{h+1}^n \frac{x_i-k+2}{k-1}\binom{x_i}{k-2}\\
&\leq \frac{x_{h+1} - k +2 }{k - 1} \sum_{i = h + 1}^{n} \binom{x_i}{k - 2} \\
&\leq \left( \frac{m}{h+1 } \right)^{ 1 / (k - 2)} \frac{nm}{k - 1} \binom{n}{k - 2} \quad \quad \mbox{(by \eqref{shadow-bound} and \eqref{xh-bound})}\\
\end{align*} 
By our choice of $h$, $h+1\geq (\frac{5m}{\ve})^{k-1}$. 
So we have
$\sum_{i>h} |\A_i|<\left(\frac{\ve}{5}\right)^{1 + \frac{1}{k - 2}}\frac{n}{k-1}\binom{n}{k-2}\leq \frac{\ve}{4} a\binom{n}{k-1}$, for sufficiently large $n$ and by our assumption $a \geq \ve^{\frac{1}{ k - 2}}$.
Let $W := [h]$. Let $\G_1:= \{ F \in \F^*: c(F) \not \in W\}$ and $\G_2:= \{F \in \F^*: |F \cap W| \geq 2\}$.

Then by the argument above, $|\G_1|=\sum_{i>h}|\A_i| \leq \frac{\ve}{4} a \binom{n}{k - 1}$. Also, for $n$ sufficiently large, 
we have $|\G_2| \leq \binom{W}{2} \binom{n}{ k - 2}  \leq \frac{\ve}{4} a\binom{n}{k-1}$.
Let $\F'':= \F^* \setminus (\G_1 \cup \G_2)$. We can
readily check that $\F''$ satisfy all of conditions 1,2,3. 
\end{proof}

Let us close this section by noting that a structural dichotomy like Theorem \ref{centralization}
was used in earlier works such as \cite{Furedi-trees, FJ-cycles, FJ-trees}. However, in those papers,
such a dichotomy was made possible only under the extra assumption that $\F$ is $\HH$-free, where $\HH$ is
a $2$-contractible tree. Here, Theorem \ref{centralization} is applicable to any $k$-graph with $\Theta(n^{k-1})$
edges. Since we will aim to count number of copies of $\HH$ in a hosting graph $\F$, the removal of the
$\HH$-free condition is essential for our overall arguments.

\section{Optimal Supersaturation at the Tur\'an threshold} \label{sec:supersaturation}

In this section, we develop an asymptotically tight bound on the number of copies of a $2$-contractible tree $\HH$
in a host graph with size just beyond the Tur\'an threshold for $\HH$, which provides the most important
foundation on which the rest of our work is built on. It may also be of independent interest as a supersaturation result.
We start with the following two folklore lemmas. Recall the definition of $\delta_i(\F)$ from
Section \ref{sec:notation}.

\begin{lemma}~\label{lem:mindegree}
Let $k\geq 2$ be an integer.
Let $\F$ be a $k$-graph. Then $\F$ contains a nonempty subgraph $\F'$ with at least $\frac{1}{2} |\F|$ edges such that
for each $i=1,\dots, k-1$, $\delta_i(\F')\geq
\frac{1}{2k}\frac{|\F|}{\binom{n}{i}}$. 
\end{lemma}

\begin{proof}
We iteratively remove edges of $\F$ to form $\F'$ in accordance with the following process. 
Whenever there is some $B \subseteq V(\F)$ with $|B|\leq k-1$ such that $B$ is contained in at least one  but fewer than  $\frac{1}{2 k } \frac{|\F|}{\binom{n}{|B|}}$ edges, we mark $B$ and remove all edges containing $B$ from $\F'$. 
Clearly at some point this process will terminate. We denote the final graph by $\F'$.

Let us count now how many edges $\F'$ has. For each $i=1,\dots, k-1$, there are at most $\binom{n}{i}$ marked $i$-sets.
For each marked $i$-set, we have removed at most $\frac{1}{2k}\frac{|\F|}{\binom{n}{i}}$ edges from $\F$.
 Summing over $i=1, \dots, k-1$, we see that we have removed at most $\frac{1}{2} |\F|$ edges. Hence,
 $|\F'| \geq \frac{1}{2} |\F|$. Since the process terminates with a nonempty $\F'$, it must be the case that
 for each $i=1,\dots, k-1$ and any $i$-set contained in an edge of $\F'$ it is in at least $\frac{1}{2k}\frac{|\F|}{\binom{n}{i}}$
 edges, completing the proof.
\end{proof}

\begin{lemma}~\label{lem:embedding}
Let $k,s,t$ be positive integers, where $k,s\geq 2$, $t\geq 1$.
Let $\Hc$ be a $k$-tree on $s$ vertices with $t$ edges. 
Let $\F$ be a $k$-graph on $[n]$, where $|\F|=a \binom{n}{k - 1}$ and $a \geq 8sk !$. Then there exists a constant $\eta>0$ such that $\F$ contains at least $\eta a^{t} n^{s - t}$ copies of $\HH$. 
\end{lemma}
\begin{proof}
By picking $\eta$ sufficiently small, we may assume $n$ is sufficiently large in terms of $s$ and  $k$. By Lemma \ref{lem:mindegree}, $\F$ contains a subgraph $\F'$ such that $|\F'|\geq \frac{1}{2}|\F|$ 
and for each $i\in [k-1]$, $\delta_i(\F')\geq \frac{1}{2k} \frac{|\F|}{\binom{n}{i}}=\frac{a}{2k}\frac{\binom{n}{k-1}}{\binom{n}{i}}$. 
Let $E_1,\dots, E_t$ be an
tree-defining ordering of $\Hc$. For each $j\in [t]$, let $\Hc_j=\{E_1,\dots, E_j\}$.
We prove by induction that for each $j=1,\dots, t$, $\F$ contains at least $(\frac{a}{8k!})^jn^{v(\HH_j)-j}$ copies of $\HH_j$.

For the base case, $\F'$ contains at least $\frac{1}{2}|\F|=\frac{a}{2}\binom{n}{k-1}\geq \frac{a}{8k!} n^{v(\HH_1)-1}$ copies of $\HH_1$.
So the claim holds.
For the induction step, let $1\leq j\leq t-1$ and suppose $\F'$ contains at least $(\frac{a}{8k!})^j n^{v(\HH_j)-j} $ copies of $\HH_j$.
Let $F$ denote a parent edge of $E_{j+1}$ in $\HH_j$. Let $L=F\cap E_{j+1}=V(\HH_j)\cap E_{j+1}$ and $\ell=|L|$.
Let $\HH'$ denote a copy of $\HH_j$ in $\F'$ and $L'$ the image of $L$ in $\HH'$.
By our assumptions of $\F'$, $\deg_{\F'}(L') \geq \frac{a}{2k} \frac{\binom{n}{k - 1}}{\binom{n}{\ell}}$. For each vertex $v \in V(\HH')$, we have that $v$ is contained in at most $\binom{n - \ell - 1}{k - \ell - 1}$ edges which contain $L'$. As there are at most $s$ vertices 
in $\HH_j$, we have that for sufficiently large $n$ the number of edges of $\F'$ that contain $L'$ but do not intersect $\HH'$ somewhere else is at least
\begin{align*}
\frac{a}{2k} \frac{\binom{n}{k - 1}}{\binom{n}{\ell}} - s \binom{n - \ell - 1}{k - \ell - 1} \geq 
\frac{a \ell !}{4 k!} n^{k - 1 - \ell} - \frac{s}{(k - \ell - 1)!} n^{k - \ell - 1}
\geq \frac{a}{8k!}n^{k - 1 - \ell},
\end{align*}
where we used the fact that $a\geq 8sk!$.
Hence $\HH'$ can be extended to at least $\frac{a}{8k!} n^{k-1-\ell}$ copies of $\HH_{j+1}$ in $\F'$.
Therefore, $\F'$ contains at least $(\frac{a}{8k!})^{j+1}n^{v(\HH_j)-j+k-1-\ell}=\frac{a}{8k!}^{j+1}n^{v(\HH_{j+1})-
(j+1)}$ copies of $\HH_{j+1}$. This completes the induction and the proof. 
\end{proof}

Before we prove the supersaturation theorem, we first deal with one of the cases.

\begin{lemma}~\label{lem:caseone}
Let $k\geq 4$ be an integer.
Let $\HH$ be a $2$-contractible $k$-tree with $s$ vertices and $t$ edges. 
Let $\F$ be an $s$-super-homogeneous $k$-graph on $[n]$ 
such that $|\F|\geq \ve n^{k - 1}$ and $\MI(\F)$ is of type 1.
Then, $\F$ contains at least $(\frac{\ve}{2ks})^t n^{s - t}$ many copies of $\HH$. 
\end{lemma}
\begin{proof}
Let $(X_1,\dots, X_k)$ be a $k$-partition associated with $\MI(\F)$. Without loss of generality, 
suppose $2^{[k-2]}\subseteq \MI(\F)$. Since $\HH$ is $2$-contractible, each edge $E$ contains at least two vertices of degree $1$.
We will designate two of them as {\it expansion vertices} for $E$.

Let $E_1,\cdots, E_t$ be a tree-defining ordering  of $\HH$. For each $i\in [t]$, let $\HH_i=\{E_1,\dots, E_i\}$,
$v_i=|V(\HH_i)|$, and let $\Pi_i$ denote the set of injective homomorphisms $\varphi^i$ of $\HH_i$ into $\F$ that
maps all the expansion vertices in $\HH_i$ into $X_{k-1}\cup X_k$.
We will prove by induction the stronger statement that for each $i\in [t]$,  $|\Pi_i|\geq (\frac{\ve}{2ks})^i n^{v_i-i}$. The base step $i=1$ is trivial
since $\F$ has at least $\ve n^{k-1}$ edges $E$ and for each $E$ we can define an embedding $\varphi^1$ of $\HH_1$ to $E$
where the expansion vertices are mapped to $X_{k-1}\cup X_k$.
For the induction step, let $1\leq i\leq t-1$ and suppose $|\Pi_i|\geq (\frac{\ve}{2ks})^i n^{v_i-i}$.
Let $F$ be a parent edge of $E_{i+1}$ in $\HH_i$. Consider any $\varphi^i\in \Pi_i$.
Since vertices in $F\cap E_{i+1}$ have degree at least two in $\HH$ and are not expansion vertices in $\HH_i$,
we have $\varphi^i(F\cap E_{i+1})\subseteq \bigcup_{j=1}^{k-2} X_j$. Since $2^{[k-2]}\subseteq \MI(\F)$,
we have $\pi (\varphi^i(F\cap E_{i+1})) \in \MI(\F)$. So,  $\Lc_\F(\varphi^i(F \cap E_{i+1}))$ is $s$-diverse. Therefore, there are at least $\frac{s -v_i}{s} | \Lc_\F( \varphi^i (F\cap E_{i+1}))|\geq\frac{1}{s} | \Lc_\F( \varphi^i (F\cap E_{i+1}))|$ edges of $\F$ that
contain $\varphi^i(F\cap E_{i+1})$ and do not intersect the rest of $\varphi^i(\HH_i)$. 
Let $\ell=|F\cap E_{i+1}|$. Since $\F$ is $s$-super-homogeneous with $\MI(\F)\supseteq 2^{[k-2]}$
and $\pi(\varphi^i(F\cap E_{i+1}))\in \MI(\F)$, by the conditions of $s$-super-homogeneous given in Theorem \ref{homogeneous},
$|\Lc(\varphi^i (F \cap E_{i+1}))| \geq \frac{|\F|}{2kn^\ell} \geq \frac{\ve}{2k} n^{ k - \ell-1}$. Hence each mapping
$\varphi_i$ in $\Pi_i$ can be extended to at least $\frac{\ve}{2ks} n^{ k -\ell- 1}$ 
injective mappings $\varphi^{i+1}$ of $\HH_{i+1}$, by mapping $E_{i+1}$ to an edge $F'$ of $\F$
that contains $\varphi^i(F\cap E_{i+1})$ and avoids the rest of $\varphi^i(\HH_i)$. Furthermore,
we can require $\varphi^{i+1}$ to map the two designated expansion vertices of $E_{i+1}$ to $X_{k-1}\cup X_k$.
Hence, $|\Pi_{i+1}|\geq |\Pi_i|\cdot \frac{\ve}{2ks} n^{ k -\ell- 1}
\geq (\frac{\ve}{2ks})^{i+1} n^{v_{i+1}-(i+1)}$.  This completes the induction step and the proof.
\end{proof}

We are now ready to prove our first supersaturation theorem. 

\begin{theorem} \label{supersat}
Let $k\geq 4$ be an integer.
Let $0<\gamma<1$, and let $\HH$ be a $2$-contractible $k$-tree with $s$ vertices and $t$ edges.
Then, there exists $\beta > 0$, depending on $\HH$ and $\gamma$, such that the following holds. Let $\F$ be a $k$-graph on $[n]$ with
$|\F|= a  \binom{n}{k - 1}$, where $a \geq \sigma(\HH) - 1 + \gamma$. Then,  $\F$ contains at least $\beta a^t n^{s - t}$ many copies of $\HH$. 
\end{theorem}

\begin{proof}
For convenience, let $\sigma=\sigma(\HH)$. Note that if $a \geq 8 s k!$, we may apply Lemma~\ref{lem:embedding}, so at the cost of a constant factor on $\beta$, we will assume $a = \sigma - 1 + \gamma$. 
We apply Theorem~\ref{centralization} to $\F$ with $\ve = \frac{\gamma}{4 \sigma}$ if $\sigma \geq 2$ and $\ve = \frac{1}{4} \gamma^{k - 2}$ if $\sigma = 1$.
Then one of the following two cases holds.

\medskip

\noindent {\bf Case 1.} $\F$ contains an $s$-superhomogeneous subgraph $\F'$ such that $|\F'|\geq \frac{1}{2} c(k,s)\ve |\F|$ and
$\MI(\F')$ is of type 1.

By Lemma~\ref{lem:caseone}, $\F'$ contains $\left( \frac{\ve c(k, s) }{2 ks}\right)^t n^{s - t}$ copies of $\T$. 
So the claim holds.

\bigskip

\noindent {\bf Case 2.} There exists a $W \subseteq [n]$ and $\F'' \subseteq \F$ satisfying the following: 

\begin{enumerate}
\item[(a)] $$|W| \leq \left( \frac{160 \sigma^2}{c(k,s) \gamma^{2k - 4}} \right)^{k - 1}.$$
\item[(b)] $$|\F''| \geq \left(1-\frac{\gamma}{4\sigma}\right)|\F|.$$
\item[(c)] Every $F \in \F''$ satisfies $|W \cap F| = 1$. 
\end{enumerate}

Observe that in $\sigma = 1$ case, Theorem~\ref{centralization} gives the upper bound $W \leq \left( \frac{160}{c(k, s) \gamma^{ 2k - 4} } \right)^{ k -1}$ and in $\sigma \geq 2$ case, gives the bound $W \leq \left( \frac{160 \sigma^2}{c(k, s) \gamma^{ 2} } \right)^{ k -1} $. In both cases, the written upper bound holds. For simplicity, observe that for both choices of $\ve$, $|\F''| \geq (1 - \ve)|\F| \geq  \left( 1 - \frac{\gamma}{4 \sigma}\right) |\F|$. 

For each $S\in \binom{W}{\sigma}$, 
let $\Lc^*(S) = \bigcap_{v \in S} \Lc_{\F''}(v)$.
We say a set $S \in \binom{W}{\sigma}$ is {\it good}, if $|\Lc^*(S)| \geq 8s (k - 1)! \binom{n}{k-2}$;
otherwise we say $S$ is {\it bad}. Let $\D$ denote the set of $(k-1)$-sets $D$ in $[n]\setminus W$
with $d_{\F''}(D)\geq \sigma$. Since each edge of $\F''$ contains one vertex in $W$ and a $(k-1)$-set in $[n]\setminus W$,
via double counting, we get
\begin{equation}\label{center-double-count}
\sum_{S \in \binom{W}{\sigma}} |\Lc^*(S)| = \sum_{D \in \D} \binom{d_{\F''}(D)}{\sigma}.
\end{equation}

By our assumption, $|\F''|\geq (1-\frac{\gamma}{4 \sigma})|\F|\geq (1-\frac{\gamma}{4\sigma})(\sigma-1+\gamma)\binom{n}{k-1}\geq
(\sigma-1+\frac{\gamma}{2})\binom{n}{k-1}$. This implies that $m:=\sum_{D\in \D} d_{\F''} (D)\geq \frac{\gamma}{2}\binom{n}{k-1}$.
Observe that for all $D \in \D$, $\binom{d_{\F''}(D)}{\sigma} \geq \frac{d_{\F''}(D)}{\sigma} $. 

Thus, $$\sum_{S \in \binom{W}{\sigma}} |\Lc^*(S)| \geq \frac{m}{\sigma} \geq \frac{\gamma}{2 \sigma }\binom{n}{ k - 1}.$$


On the other hand, the contribution to the left-hand side of \eqref{center-double-count} from the bad $\sigma$-sets of $W$
is at most $\binom{|W|}{\sigma} 8s(k-1)!\binom{n}{k-2}<\frac{\gamma}{4\sigma}\binom{n}{k-1}$, for sufficiently large $n$.
Hence, we have
\[\sum_{S\in \binom{W}{\sigma}, \, S \text{ good}} |\Lc^*(S)| \geq \frac{\gamma}{4\sigma} \binom{n}{ k - 1}.\]

Fix any $S\in \binom{|W|}{\sigma}$ that is  good, let $d(S) = \frac{|\Lc^*(S)|}{\binom{n}{ k - 2}}$, in particular, since $S$ is good, $d(S) \geq 8s (k - 1)!$. 
Let $R$ be a minimum cross-cut of $\HH$. Let $\HH'=\HH-R$. It is easy to see that $\HH'$ is a $(k-1)$-tree.
Observe that any copy of $\HH'$ in $\Lc^*(S)$ corresponds to a copy of $\HH$ in $\F$ by joining it to $S$
and different copies of $\HH'$ in $\Lc^*(S)$ give rise to different copies of $\HH$.
Let $t'=|\HH'|$ and note that $v(\HH')=s-\sigma$.  By Lemma~\ref{lem:embedding}, for every good $S$, there are at least $\eta d(S)^{t'}n^{s - \cross - t'}$ many copies of $\HH'$ in $\Lc^*(S)$. Thus, the number of copies of $\HH$ in $\F$ is at least: 

\begin{align*}
    \sum_{S\in \binom{W}{\sigma}, \, S \text{ good }} \eta d(S)^{t'} n^{s - \sigma - t'} &\geq  \eta n^{s - \cross - t' - (k - 2)t'} \sum_{S\in \binom{W}{\sigma},\, S \text{ good }} |\Lc^*(S))|^{t'} \\
    &\geq  \eta n^{s - \sigma - t' - (k - 2)t'} |W|^{\sigma} \left( \frac{1}{|W|^{\sigma}} \sum_{S  \text{ good }} |\Lc^*(S)|\right) ^{t'}\\
        &\geq  \eta n^{s - \sigma - t' - (k - 2)t'} |W|^{\sigma} \left( \frac{1}{|W|^{\sigma}} \frac{\gamma}{4 \sigma} \binom{n}{ k - 1}\right) ^{t'}\\
    &\geq \beta  a^t n^{s - \sigma}\geq \beta a^t n^{s-t},
\end{align*}
for some constant $\beta>0$, where we used the fact that $a=\sigma-1+\gamma$ is a constant 
and $|W|$ is bounded by a constant.
\end{proof}

Note that the bound obtained in Theorem \ref{supersat} is tight up to a constant factor, as shown by considering $\gnpk$, with $p=|\F|/\binom{n}{k}$. 

\section{Optimal Balanced Supersaturation at the Tur\'an Thresold} \label{sec:balanced-supersaturation}

In this section, we develop a balanced version of the supersaturation theorem of Section 4. In the next section, we will use this balanced supersaturation theorem to obtain almost tight bounds on the number of $\Hc$-free graphs. 


Let $d>0$ be a real. A $t$-uniform hypergraph $\K$ is called {\it $d$-graded} if 
for all sets $S \subseteq V(\K)$ with $1\leq |S|\leq t$, we have
$$d_{\K}(S) \leq d^{t - |S|}.$$

We say $S \subseteq V(\K)$ is {\it saturated} by $\K$ if $d_\K(S) = \floor{d^{t - |S|}}$. We say $S$ is {\it admissible} if no subset of it is saturated, and {\it inadmissible} otherwise. For each admissible $S\subseteq V(\K)$, let

$$\Z_{\K}(S) := \{ v \in V(\K) : \{v \} \cup S \text{ is inadmissible and } \{v\} \text{ is not saturated}\}.$$

\begin{lemma}\label{boundedhypergraph}
Let $d \geq 2$ be a real and $\K$ be a $d$-graded $t$-uniform hypergraph of size less than $d^{t - 1} M$. Then, the set of vertices $v \in V(\K)$ satisfying $d_\K(v) = \floor{d^{t - 1}}$ is less than $2 t M$. 
Furthermore, any admissible set $S$ of size $|S| \leq t - 1$, satisfies 

$$|\Z_\K(S)| \leq  t 2^{t + 1}  d.$$
\end{lemma}
\begin{proof}

Let $B = \{ v\in V(\K): d_{\K}(v) = \floor{d^{t - 1}}\}$. Then clearly,
\[|B| \floor{d^{t - 1}} \leq  t|\K|.\]
Since $|\K|\leq d^{t-1}M$, we have $|B| \leq 2 t M $. 

Now, consider any admissible set $S\subseteq V(\K)$ with $|S|\leq t - 1$.  
For every $D \subseteq S$, let $\B(D) := \{ v \in V(\K) : D \cup \{ v\} \text{ is saturated} \}$. By definition, $\Z_\K(S) = \bigcup_{D \subseteq S, D \neq \emptyset} \B(D)$. Consider a fixed nonempty $D\subseteq S$.
Note that 
\begin{equation}\label{sat-bounds}
\sum_{v \in \B(D)}d_\K( D\cup \{v \}) \leq t d_\K(D),
\end{equation}
since each edge of $\K$ containing $D$ is counted at most $t$ times in the sum on the left-hand side.
Since $S$ is admissible, $D$ is not saturated. Thus,  $d_\K(D) \leq d^{t - |D|}$. On the other hand, for every $v \in \B(D)$, $D\cup \{v\}$ is saturated. So, 
$d_\K(D \cup \{v\}) = \floor{d^{t - |D| - 1}}$. By \eqref{sat-bounds}, we get $|\B(D)|\floor{d^{t-|D|-1}}\leq t d^{t-|D|}$, implying
$$|\B(D)| \leq 2 t d.$$

Summing over all $D \subseteq S$ with $D \neq \emptyset$, we get $|\Z(S)|\leq t2^{t+1}d$.
\end{proof}

We say that a hypergraph $\Hc$ is {\it $\ell$-overlapping} if every two edges $E, F \in \Hc$ satisfy that $|E \cap F| \leq \ell$.   
The next lemma  says for any $2$-contractible, $\ell$-overlapping $k$-tree $\HH$, a sufficiently dense $k$-graph $\F$ always contains a dense balanced collection of copies of $\HH$. 

\begin{lemma}~\label{lem:heavybalanced}
Let $k\geq 4$ be an integer. Let $\Hc$ be a  $2$-contractible, $\ell$-overlapping $k$-tree with $t$ edges. There exists $a_0$ and $\beta$ such that the following holds if $n$ is sufficiently large. Let $\F$ be a $k$-graph on $[n]$ that has size $a\binom{n}{ k - 1}$, with $a \geq a_0$. Then, there exists a collection $\K$ of copies of $\Hc$ in $\F$ such that: 
\begin{enumerate}
\item $$|\K| \geq \beta an^{ k - 1} (\beta an^{k - \ell - 1})^{t - 1}$$
\item For all $S \subseteq \F$ with $1\leq |S|\leq t$, $$d_\K(S) \leq  (\beta an^{k - \ell - 1})^{t - |S|}.$$
\end{enumerate}
\end{lemma}

\begin{proof}
Let $\beta$ be sufficiently small and $a_0$ be sufficiently large. 
Let $\K$ be maximal collection of copies of $\HH$ in $\F$ that satisfies condition 2. Since any collection
that consists of  a single copy of $\HH$ satisfies condition 2,
$\K$ exists. We show that $\K$ must also satisfy condition 1. 
Suppose on the contrary that 
$|\K| < \beta^t n^{ k - 1} (a n^{k - \ell})^{t - 1}$. To derive a contradiction, it suffices to show
that $\F$ contains a copy $\HH^*$ of $\HH$ that is admissible relative to $\K$. Indeed,
if we found such an  $\HH^*$, in particular, it would not be saturated by $\K$. Hence,
$d_{\K}(\HH^*) < (\beta a n^{k-\ell-1})^{t-t}=1$. So, $\HH^*\notin \K$. 
Furthermore, for all $S \subseteq \HH^*$, we would have  $d_{\K}(S) \leq  \floor{(\beta a n^{k - \ell})^{t - |S|}} - 1$. Thus,
$d_{\K\cup\HH^*}(S)\leq (\beta a n^{k - \ell})^{t - |S|}$.  
Therefore, $\K\cup \{\HH^*\}$ would still satisfy 
condition 2, which would contradict the maximality of $\K$ and complete the proof. Thus, it suffices to find such an admissible $\Hc^*$.

Since $\K$ satisfies condition 2,
$\K$ is $\beta a n^{ k- \ell - 1}$-graded.  Recall that a set $S \subseteq \F$ is saturated by $\K$ if $d_\K(S) =  \floor{(\beta an^{k - \ell})^{t - |S|}}$. Let $\F_{\mathrm{heavy}} := \{F \in \F: d_\K(F) = \floor{(\beta an^{k - \ell})^{t - 1}}\}$, 
so that $\F_{\mathrm{heavy}}$ consists of all the saturated edges of $\F$. By Lemma~\ref{boundedhypergraph}, $|\F_{\mathrm{heavy}}| \leq 2t \beta a n^{ k - 1}$.
Let $\F^* = \F \setminus \F_{\mathrm{heavy}}$. Then $|\F^*|\geq \frac{1}{2}|\F|$, as long as  $\beta < \frac{1}{8 t (k - 1)!}$.
By definition, every edge in $\F^*$ is admissible relative to $\K$.
By Lemma~\ref{lem:mindegree}, there exists $\F' \subseteq \F^*$ such that $|\F'| \geq \frac{1}{4} |\F|$ and 
for each $i\in [k-1]$, $\delta_i(\F')\geq \frac{a}{4 k } \frac{\binom{n}{k - 1}}{\binom{n}{ i}}$.  

For each $i\in [t]$, let $\HH_i=\{E_1,\dots, E_i\}$. We will inductively find  embeddings $\varphi^i$ of $\HH_i$ into $\F'$, such
that $\varphi^i(\HH_i)$ is admissible relative to $\K$.
First, let $F_1$ be any edge of $\F'$ and let $\varphi^1$ be any bijection from the vertices of $E_1$
to the vertices of $F_1$. Since $F_1$ is admissible relative to $\K$,
$\varphi^1(\HH_1)$ is admissible.
Let $1\leq i\leq t-1$ and suppose we have defined an embedding
$\varphi^i$ of $\HH_i$ into $\F'$ such that $\varphi^i(\HH_i)$ admissible relative to $\K$. 
Let $F$ denote a parent edge of $E_{i+1}$ in $\HH_i$.
Let $p=|F\cap E_{i+1}|$.
Since $\HH$ is $\ell$-overlapping, $p\leq \ell$. By our conditions on $\F'$, we have
$d_{\F'}(\varphi^i(F\cap E_{i + 1})) \geq \frac{a}{4k} \frac{\binom{n}{k - 1}}{\binom{n}{p}}$. The number of edges 
of $\F'$ that contain $\varphi^i(F\cap E_{i + 1})$ and some other vertex of $\varphi^i(\HH_i)$ is at most $kt \binom{n}{ k - p - 1}$. Since $\K$ is $\beta a n^{ k- \ell - 1}$-graded,
by Lemma~\ref{boundedhypergraph}, we have $\Z_\K(\varphi^i(\HH_i)),$ the set of edges $E$ of $\F'$ such that $\{E\} \cup \varphi^i(\HH_i)$ is inadmissible relative to $\K$, has size at most $t 2^{t + 1} \beta a n^{ k -\ell - 1}$. 

Therefore as long as  $\beta$ and $a$ satisfy the following inequality: 

$$\frac{a}{4 k} \frac{\binom{n}{k - 1}}{\binom{n}{p}} > kt \binom{n}{ k - p - 1} + 2^{ t + 1} t \beta a n^{ k - \ell - 1},$$

there is some edge $E$ of $\F'\setminus \Z_\K(\varphi^i(\HH_i))$ that contains $\varphi^i(F\cap E_{i+1})$ but does not intersect the rest of $\varphi^i(\HH_i)$.
Picking some $a_0$ sufficiently large in terms of $k, \ell, t$ and $\beta$ sufficiently small in terms of $k, \ell, t$, the inequality holds and such $E$ exists. We extend $\varphi^i$ to an embedding $\varphi^{i+1}$ of $\HH_{i+1}$ into $\F'$ by mapping the vertices of
$E_{i+1}\setminus (F\cap E_{i+1})$ injectively to the vertices of $E\setminus \varphi^i (F\cap E_{i+1})$ so that
$\varphi^{i+1}(E_{i+1})=E$. Note that $\varphi^{i+1}(\HH_{i+1})=\varphi^i(\HH_i)\cup \{E\}$. Since $E \not \in \Z_\K(\varphi^i(\HH_i))$ and $\varphi^{i + 1}(\HH_{i+1})$ is admissible relative to $\K$. This completes the induction. Now, $\varphi^t(\HH_t)$ is
a copy of $\HH$ that is admissible relative to $\K$, completing our proof.
\end{proof}

Next, we seek to prove a similar result for $|\F| = a \binom{n}{k - 1}$ with $\sigma(\HH) - 1 + \gamma \leq a \leq a_0$,
where $\gamma$ is any given small positive real. At the cost of a constant, we will prove a balanced supersaturation lemma only for $a = \sigma(\HH) - 1  + \gamma$. 

\begin{lemma}\label{lem:balancedlight}
Let $k\geq 4$ be an integer.
Let $\HH$ be a $2$-contractible $\ell$-overlapping $k$-tree with $t$ edges. Let $\gamma > 0$ be a positive real. There exists a $\beta > 0$ such that the following holds if $n$ is sufficiently large. Let $\F$ be a $k$-graph on $[n]$ with  $|\F|\geq (\sigma(\HH) -1 + \gamma) \binom{n}{ k - 1}$. Then, there exists a collection $\K$ of copies of $\T$ in $\F$ such that: 
\begin{enumerate}
 \item $$|\K| \geq \beta n^{ k - 1} (\beta n^{k - \ell  - 1})^{t - 1}$$
\item For all $S \subseteq \F$ with $1\leq |S|\leq t$, 
\[d_\K(S) \leq  (\beta n^{k - \ell - 1})^{t - |S|}.\]
\end{enumerate}
\end{lemma}

\begin{proof}
For convenience, let $\sigma=\sigma(\HH)$.
Let $\beta$ be sufficiently small and $a_0$ be sufficiently large. 
Let $\K$ be maximal collection of copies of $\HH$ in $\F$ that satisfies condition 2. Since any collection consisting of
a single copy of $\HH$ satisfies condition 2,
$\K$ exists. We show that $\K$ must also satisfy condition 1. 
Suppose on the contrary that 
$|\K| < \beta^t n^{ k - 1} (n^{k - \ell})^{t - 1}$. To derive a contradiction, as in the proof Lemma~\ref{lem:heavybalanced}, it suffices to show
that $\F$ contains a copy $\HH^*$ of $\HH$ that is admissible relative to $\K$. 

By definition, $\K$ is $\beta  n^{ k- \ell - 1}$-graded.  Recall that a set $S \subseteq \F$ is saturated with respect to $\K$ if $\deg_\K(S) =  \floor{(\beta n^{k - \ell})^{t - |S|}}$. Let $\F_{\mathrm{heavy}} := \{F \in \F: \deg_\K(F) = \floor{(\beta n^{k - \ell})^{t - 1}}\}$. Note that $\F_{\mathrm{heavy}}$ consists of all the saturated edges. By Lemma~\ref{boundedhypergraph}, $|\F_{\mathrm{heavy}}| \leq 2t \beta  n^{ k - 1}<\frac{\gamma}{2} \binom{n}{k-1}$, if $\beta$ is sufficiently small.
Let $\widehat{\F} := \F - \F_{\mathrm{heavy}}$.  Then $|\widehat{\F}| \geq (\cross -1 + \frac{\gamma}{2}) \binom{n}{ k - 1}.$ 
By definition, all edges of $\widehat{\F}$ are admissible relative to $\K$.
Let $s = v(\Hc)$. Now we apply Theorem~\ref{centralization} to $\widehat{\F}$ with 
\[a = (\cross - 1 + \frac{\gamma}{2}) \mbox{ and  }\ve = \frac{1}{\sigma} \left(\frac{\gamma}{2}\right)^{ k - 2} . \] 

There are two cases to consider:

\medskip

\noindent {\bf Case 1.} $\widehat{\F}$ satisfies case (1) of Theorem~\ref{centralization}.

\medskip

In this case, we find an $\F' \subseteq \widehat{\F}$ which is $s$-super-homogeneous, $|\F'|\geq \frac{c(k, s) \gamma^{ k -2}}{2^{k - 1} \sigma} \binom{n}{ k - 1}$, and $\MI(\F')$ is of type 1, where without loss of generality $2^{[ k - 2]} \subseteq \MI(\F')$.
Let $(X_1,\dots, X_k)$ be the $k$-partition associated with $\MI(\F')$. Since $\T$ is $2$-contractible, each edge contains at least two vertices of degree $1$.
We will designate two of them as {\it expansion vertices} for $E$.

Let $E_1,\cdots, E_t$ be a tree-defining ordering of the edges of $\HH$. For each $i\in [t]$, let $\T_i=\{E_1,\dots, E_i\}$.
Let $\Pi_i$ denote the set of embeddings $\varphi^i$ of $\HH_i$ into  $\F'$ satisfying that 
$\varphi^i(\HH_i)$ is admissible relative to $\K$ and $\varphi^i$ maps all expansion vertices in $\HH_i$ to $X_{k-1}\cup X_k$.
We use induction on $i$ to prove that for each $i\in [t]$, $\Pi_i\neq \emptyset$.

For the base step, we let $F_1$ be any edge of $\F'$. Let $\varphi^1$ be a bijection that maps the vertices of $E_1$ to the vertices
of $F_1$ such that the two expansion vertices are mapped to $F_1\cap X_{k-1}$ and $F_1\cap X_k$, respectively.
Since all edges of $\F'$ are admissible, $\varphi^1(\HH_1)$ is admissible. 
Hence, $\Pi_1\neq \emptyset$.

For the induction step, let $1\leq i\leq t-1$ and suppose that $\Pi_i\neq \emptyset$.
Let $\varphi^i\in \Pi_i$. 
Let $F$ be a parent edge of $E_{i+1}$ in $\HH_i$. Since vertices in $F\cap E_{i+1}$ have degree at least two in $\HH$, 
they are not expansion vertices in $\HH_i$.
Hence, $\varphi^i(F\cap E_{i+1})\subseteq \bigcup_{j=1}^{k-2} X_j$. Since $2^{[k-2]}\subseteq \MI(\F)$,
we have $\pi (\varphi^i(F \cap E_{i+1})) \in \MI(\F)$. So, in particular that $\Lc_\F(\varphi^i(F \cap E_{i+1}))$ is $s$-diverse. Therefore, we have that at least $1 - \frac{v(H_i)}{s} \geq \frac{1}{s}$ proportion of the edges in $\Lc_\F( \varphi^i(F \cap E_i))$  do not 
intersect the rest of $\varphi^i(\HH_i)$. Furthermore, by Lemma~\ref{boundedhypergraph}, with $d=\beta n^{k-\ell-1}$,  $|\Z_\K(\varphi^i(\HH_i))| \leq t 2^{ t + 1} \beta n^{ k - \ell - 1}$. Let $p=|F\cap E_{i+1}|$.
Since $\HH$ is $\ell$-overlapping, $p\leq \ell$. Since $\F'$ is $s$-super-homogoneus and 
$\pi(F\cap E_{i+1})\in \MI(\F')$,
\[|\Lc_{\F'}(\varphi^i(F \cap E_{i+1}))|\geq \frac{|\F'|}{2kn^p} \geq \frac{c(k,s) \gamma^{k - 2} }{2^{k} \sigma k!} n^{ k - p - 1 }
\geq \frac{c(k,s) \gamma^{k - 2}}{2^{k} \sigma k!} n^{ k - \ell - 1 }.\]
By picking $\beta$ sufficiently small in terms of $\gamma$ and $\T$, we can ensure that 
\[\frac{1}{s} | \Lc_\F'( \varphi^i(F \cap E_{i+1}))|>  |\Z_\K(\varphi^i(\HH_i))|.\]
Hence, $\F'\setminus\Z(\varphi^i(\HH_i))$  contains an edge $E$ that contains $\varphi^i(F\cap E_{i+1})$ but
does not intersect the rest of $\varphi^i(\HH_i)$.
We extend $\varphi^i$ to an embedding of $\HH_{i+1}$ in $\F'$ by mapping the vertices of $E_{i+1}\setminus (F\cap E_{i+1})$
injectively into $E\setminus \varphi^i(F\cap E_{i+1})$, so that
$\varphi^{i+1}(E_{i+1})=E$ and the expansion vertices  of $E$ are mapped into $X_{k - 1}$ and $X_k$. Since $E \not \in \Z_\K(\varphi^i(\HH_i))$,  $\varphi^{i + 1}(\HH_{i + 1})$ is admissible
relative to $\K$. This completes the induction. Now, $\varphi^t(\HH_t)$ is a copy of $\HH$ in $\F'$ that is
admissible relative to $\K$. This completes the proof for Case 1.

\bigskip
\noindent{\bf Case 2.} $\widehat{\F}$ satisfies case (2) of Theorem \ref{centralization}.

\medskip

In this case, there exist a $W \subseteq [n]$ and $\F'' \subseteq \widehat{\F}$ satisfying 

\begin{enumerate}
\item[\qquad (a)] $|W| \leq \left( \frac{5\cdot 2^{2k - 2} \sigma^2}{c(k,s) \gamma^{2k - 4}} \right)^{k - 1}.$
\item[\qquad(b)] $|\F''| \geq (1-\frac{\gamma^{k - 2}}{ 2^{ k - 2} \cdot \sigma} )|\widehat{\F}|\geq (1-\frac{\gamma}{4\sigma})(\sigma-1+\frac{\gamma}{2})
\binom{n}{k-1} \geq 
(\cross - 1 + \frac{\gamma}{8} ) \binom{n}{k - 1} .$
\item[\qquad(c)] Every $F \in \F''$ satisfies $|W \cap F| = 1$. 
\end{enumerate}

For each $S\in \binom{W}{\sigma}$, 
let $\Lc^*(S) = \bigcap_{v \in S} \Lc_{\F''}(v)$.
Let $\D$ denote the set of $(k-1)$-sets $D$ in $[n]\setminus W$
with $d_{\F''}(D)\geq \sigma$. Since each edge of $\F''$ contains one vertex in $W$ and a $(k-1)$-set in $[n]\setminus W$,
via double counting, we get

\[\sum_{S \in \binom{W}{\sigma}} |\Lc^*(S)| = \sum_{D \in \D} \binom{d_{\F''}(D)}{\sigma}.\]

By our assumption, $|\F''|\geq (\sigma-1+\frac{\gamma}{8})\binom{n}{k-1}$. This implies that $m:=\sum_{D\in \D} d_{\F''} (D)\geq \frac{\gamma}{8}\binom{n}{k-1}$.
Thus, by observing that if $d_{\F''}(D) \geq \sigma$, $\binom{d_{\F''}(D)}{\sigma} \geq \frac{d_{\F''}(D)}{\sigma}$, the right-hand side is at least 
$\frac{\gamma}{8 \sigma} \binom{n}{k -1}$. 



Hence, there exists a  $S \in \binom{W}{\sigma}$, such that 
\begin{equation} \label{Ls-size}
|\Lc^*(S)| \geq\frac{\gamma}{8 \sigma |W|^\sigma} \binom{n}{ k - 1}.
\end{equation}

 We fix such an $S$. By Lemma \ref{lem:mindegree}, $\Lc^*(S)$ contains a subgraph $\Lc'$
 such that $|\Lc'|\geq \frac{1}{2}|\Lc^*(S)|$ and  for each $i\in [k-1]$, 
 $\delta_i(\Lc')\geq \frac{1}{2k}\frac{|\Lc^*(S)|}{\binom{n}{i}}$. Note that since $\Lc'\subseteq \Lc^*(S)$,
 for each $D\in \Lc'$ and each $v\in S$, $D\cup \{v\}\in \F''$.

 Let $E_1,\dots, E_t$ be a tree-defining order of $\HH$.
 For each $i\in [t]$, let $\HH_i=\{E_1,\dots, E_i\}$.  Let $R$ be a minimum cross-cut of $\HH$. 
 For each $i\in [t]$, let $\{v_i\}=E_i\cap R$ and $E'_i=E_i\setminus \{v_i\}$.
 Let $\HH' = \HH - R$. It is easy to see by definition that $\HH'$ is a $(k-1)$-tree for which $E'_1,\dots, E'_t$
 is a tree-defining ordering of $\HH'$. Furthermore, since each $E_i$ contains two degree $1$ vertices, $E'_1,\dots, E'_t$
 are all distinct.

For each $i\in [t]$, let $\Pi_i$ denote the collection of embeddings of $\varphi^i$ of $\HH_i$ into $\F'$ such that 
$\varphi(\HH_i)$ is admissible relative to $\K$. We prove by induction on $i$ that $\Pi_i\neq \emptyset$.
Let $g$ be any bijection from $R$ to $S$. For the base step, let $D_1$ be any edge of $\Lc'$. We define $\varphi^1$
to be any mapping from $V(\HH_1)$ to $V(\F'')$ that maps $v_1$ to $g(v_1)$ and $E'_1$ to $D_1$.
By earlier discussion, $\varphi^1(E_1)= D_1\cup \{v_1\}$ is an edge in $\F''$.
Since all edges in $\widehat{\F}$ are admissible, all edges in $\F''$ are admissible. Hence, $\varphi^1(\HH_1)=
\varphi^1(E_1)$ is admissible relative to $\K$. For the induction step, let $1\leq i\leq t-1$ and suppose $\Pi_i\neq \emptyset$.
Let $\varphi^i\in \Pi_i$. Suppose a parent edge of $E'_{i+1}$ in $\HH'_i$ is $E'_j$, for some $j\leq i$.
Let $p=|E_j\cap E'_{i+1}|$.
Since $\HH$ is $\ell$-overlapping, $\HH'$ is $\ell$-overlapping. So $p\leq \ell$. By \eqref{Ls-size} and subsequent discussion,
$d_{\Lc'}(\varphi^i(E'_j\cap E'_{i+1}))\geq \frac{\gamma}{16\sigma |W|^\sigma}\frac{\binom{n}{k-1}}{\binom{n}{p}}$.
The number of edges of $\Lc'$  containing $\varphi^i(E'_j\cap E'_{i + 1})$ and intersect the rest of $\varphi^i(\HH_i)$
is at most $kt \binom{n}{ k - p - 2}$. 
Furthermore, 
by Lemma~\ref{boundedhypergraph}, with $d=\beta n^{k-\ell-1}$, $|\Z_\K(\varphi^i(\HH_i))|\leq 2^{ t + 1} t \beta  n^{ k - \ell -1}$. 
Let $\Z'=\{E\setminus R: E\in \Z_\K(\varphi^i(\HH_i))\}$. Then $|\Z'|\leq 2^{t+1}t\beta n^{k-\ell-1}$.
By choosing $\beta$ small enough, for sufficiently large $n$, we have

\[\frac{\gamma}{16 k \sigma |W|^\sigma }  \frac{\binom{n}{k - 1}}{\binom{n}{p}} > kt \binom{n}{ k - p  - 2} + 2^{ t + 1} t \beta  n^{ k - \ell -1},\]

where we used the fact that $|W|$ is bounded by a constant.

Hence, $\Lc'\setminus Z'$  contains an edge $D$ that contains $\varphi^i(E_j'\cap E_{i+1}')$ but
does not intersect the rest of $\varphi^i(\HH_i)$. 
We extend $\varphi^i$ to an embedding of $\HH_{i+1}$ in $\F'$ by mapping the vertices of $E'_{i+1}\setminus (E'_j\cap E'_{i+1})$
injectively into $D\setminus \varphi^i(E'_j\cap E'_{i+1})$ and mapping $v_{i+1}$ to $g(v_{i + 1})$ so that
$\varphi^{i+1}(E_{i+1})=D \cup \{ g(v_{i + 1})\}$. Since $D \not \in \Z'$, 
$\varphi^{i+1}(E_{i+1})\notin \Z_\K(\varphi^i(\HH_i))$, and hence $\varphi^{i + 1}(\HH_{i + 1})$ is admissible
relative to $\K$. This completes the induction. Now,  $\varphi^t(\HH_t)$ is a copy of $\HH$ in $\F''$ that is admissible
relative to $\K$. This completes the proof for Case 2.
\end{proof}

Lemma \ref{lem:heavybalanced} and Lemma \ref{lem:balancedlight}  together imply the following: 

\begin{theorem}\label{balancedsupersaturation}
Let $k\geq 4$.
Let $\ve > 0$. Let $\HH$ be a $2$-contractible $\ell$-overlapping $k$-tree with $t$ edges. There exist $c=c(\ve, \HH), c'=c'(\ve,\HH)>0$ such that the following holds. Let $\F$ be a $k$-graph on $[n]$ with $|\F|=a\binom{n}{k-1}$, where
$a\geq \sigma(\HH) - 1 + \ve$. Then, there exists a collection $\K$ of copies of $\T$ in $\F$ such that 
$|\K| \geq c' an^{k-1}(an^{k-\ell-1})^{t-1}$ and for all $1 \leq b \leq t$: 
$$\Delta_{b}(\K) \leq c \cdot [\tau(\F)]^{b - 1}\frac{|\K|}{|\F|},$$

where $\tau(\F) := \frac{1}{ a n^{k - \ell - 1}}$. \qed
\end{theorem}
Observe that the size of the collection $\K$ returned by Theorem \ref{balancedsupersaturation} 
matches the bound in Theorem \ref{supersat}  if the intersection of each edge of $\HH$ and its parent edge has size exactly $\ell$.
Thus, it can be viewed as a direct strengthening of Theorem \ref{supersat} for such $2$-contractible hypertrees.

\section{Number of $\HH$-free hypergraphs} \label{sec:main-proofs}
In this section, we prove Theorem \ref{main} and Theorem \ref{random-tree}. First, we 
recall the celebrated Container Lemma as follows.

\begin{lemma}[\cite{BMS, ST}]\label{lem:containerlemma}
For every $t \in \NN$ and every $c > 0$, there exists a $\delta > 0$ such that the following holds for all $N \in \NN$. Let $\K$ be a $t$-uniform hypergraph on $N$ vertices such that for all $1 \leq b \leq t$, $$\Delta_b(\K) \leq c \tau^{b - 1} \frac{|\K|}{N}.$$ Then there exists a collection $\C$ of subsets of $V(\K)$ such that the following holds:
\begin{enumerate}
\item For every $I \in \I(\K)$, there is a $C \in \C$ such that $I \subseteq C$. 
\item $|\C| \leq \binom{N}{\leq t \tau N}$. 
\item For every $C \in \C$, $|V(C)| \leq (1 - \delta)|V(\K)|$. 
\end{enumerate}
\end{lemma}

Applying Lemma \ref{lem:containerlemma} with Theorem~\ref{balancedsupersaturation}, we have the following: 
\begin{lemma}\label{container}
Let $k\geq 4$ be an integer.
Let $\T$ be a $2$-contractible $\ell$-overlapping  $k$-tree with $t$ edges and cross-cut number $\sigma$. Then, for every $\ve > 0$, there exists a $\delta >0$ such that the following holds. Let $\F$ be a $k$-graph on $[n]$ with $|\F|\geq (\sigma - 1 + \ve)\binom{n}{k - 1}$. Then, there is a collection $\C \subseteq 2^{\F}$ satisfying
\begin{enumerate}
\item For every $\T$-free subgraph $I \subseteq \F$, there is there is a $C \in \C$ such that $I \subseteq C$. 
\item $|\C| \leq \binom{\binom{n}{k}}{\leq \frac{t}{n^{k - \ell - 1}} \binom{n}{ k - 1} }$. 
\item For every $C \in \C$, $|V(C)| \leq (1 - \delta) |\F|$. 
\end{enumerate}\end{lemma}
\begin{proof}
Given $\F$, let $\K$ be the collection of copies of $\HH$ returned by Theorem~\ref{balancedsupersaturation}. 
We will view $\K$ as a $t$-uniform hypergraph with  $V(\K)=E(\F)$ whose hyperedges are the copies of $\HH$ in the collection.
By the conditon on $\K$, we have that for all $1\leq b\leq t$ 

$$\Delta_{b}(\K) \leq c \tau(\F)^{b - 1}\frac{|\K|}{|\F|}$$

where $\tau(\F) := \frac{1}{n^{k - \ell - 1}} \frac{\binom{n}{k - 1}}{|\F|}$.

Then, applying Lemma~\ref{lem:containerlemma},  there is a $\delta$ and a collection $\C \subseteq 2^{E(\F)}$ satisfying:
\begin{enumerate}
\item For every independent set $I \subseteq \Hc$, there is there is a $C \in \C$ such that $I \subseteq C$. 
\item $|\C| \leq \binom{n}{\leq \frac{t}{n^{k - \ell - 1}} \binom{n}{ k - 1} }$. 
\item For every $C \in \C$, $|V(C)| \leq (1 - \delta)|\F|$. 
\end{enumerate}
Noting that every $\T$-free subfamily of $\F$ corresponds to an independent set in $\K$, the result follows. 

\end{proof}

We now apply Lemma~\ref{container} repeatedly to derive the following lemma. 
\begin{lemma}\label{counting}
Let $k\geq 4$ be an integer.
Let $\HH$ be a $2$-contractible $\ell$-overlapping $k$-tree with $t$ edges and cross-cut number $\sigma$. Let $\ve > 0$. Then there exists a constant $K$ such that there is a collection $\C$ of size at most $$\exp\left(\frac{K \log^2(n)}{n^{k - \ell - 1}} \binom{n}{k - 1}\right)$$ subgraphs of $K_{n}^{(k)}$ such that for every $\HH$-free subgraph $I$ of $K_{n}^{(k)}$, there is a $C \in \C$ with $I \subseteq C$ and every $C \in \C$ has size at most $(\cross - 1 + \ve) \binom{n}{k - 1}$. 
\end{lemma}
\begin{proof} 
Let $\delta := \delta(\ve, \T)$ be the constant given by Lemma~\ref{container} and choose $K$ large enough depending on $\ve, \delta, t$. Let $n$ be sufficiently large.

We will build a sequence of auxiliary rooted trees $\A_1, \dots, \A_m$ such that leaves of $\A_m$ will be our containers. We begin by fixing $\A_1$ to be the tree with a single vertex, the root $K_n^{(k)}$. To construct $\A_2$, we apply Lemma~\ref{container} to $K_n^{(k)}$ to find a collection $\C$ of subgraphs of $K_n^{(k)}$. We form $\A_2$ by adding each container $C \in \C$ as a vertex and connect it to the root $K_n^{(k)}$. Observe that in $\A_2$, $K_n^{(k)}$  has no more than $$\binom{\binom{n}{k}}{\frac{2t}{n^{k - \ell - 1} }\binom{n}{k - 1}}$$ children,  every $\HH$-free subgraph of $K_n^{(k)}$ is contained in one of the children, and every child has size no more than $(1 -   \delta) \binom{n}{k - 1}$.

To find $\A_{i + 1}$ from $\A_i$, we first check to see if any leaf has size greater than $(\cross - 1 + \ve) \binom{n}{k - 1}$. If not, we terminate the process then. For every leaf $\F$ of $\A_i$ which does have size greater than $( \cross - 1 + \ve) \binom{n}{k - 1}$
(we call such $\F$ {\it dense}), we apply Lemma~\ref{container} to find a collection $\C_\F \subseteq 2^{\F}$ such that the following holds:
\begin{enumerate}
\item For every $\T$-free subgraph $I \subseteq \F$, there is there is a $C \in \C_\F$ such that $I \subseteq C$. 
\item $|\C_\F| \leq \binom{n}{\frac{2t }{n^{k - \ell - 1}} \binom{n}{ k - 1} }$. 
\item For every $C \in \C_\F$, $|V(C)| \leq (1 - \delta)|\F|$. 
\end{enumerate}

Then, from $\A_i$, we add every $C$ in $\C_\F$ as a leaf below $\F$ for each dense $\F$ to form $\A_{i+1}$.
Since for every $\T$-free graph $I \subseteq \F$, there is some $C \in \C_\F$ such that $I \subseteq C$, we have that every $\T$-free graph contained in $\F$ is contained in some child of $\F$ in $\A_{i+1}$. Furthermore, by construction, we have that there are no more than
$$ \binom{\binom{n}{k}}{\frac{2t}{n^{k - \ell - 1} }\binom{n}{k - 1}}$$ children below each such $\F$.

We will terminate this process when every leaf has size less than $(\sigma - 1+ \ve) \binom{n}{k - 1}$. Let $\A_m$ be the final such tree. Observe that (3) implies the height of $\A_m$ is no more than $O(\log n )$, (1) implies that every $\HH$-free subgraph 
of $K^{(k)}_n$  is contained in some leaf of $\A_m$, and the height condition with (2) implies that the number of leaves of $\A_m$ is no more than $$\prod_{i = 1}^{m} \binom{\binom{n}{k}}{\frac{2t}{n^{k - \ell - 1} }\binom{n}{k - 1}} \leq \exp \left( \frac{K \log^2(n)}{n^{k - \ell - 1} } \binom{n}{ k -1} \right),$$

for some large enough constant $K$ depending on $\ve, t, k, \ell$.  

\end{proof}
 We are now ready to prove Theorem \ref{main} and Theorem \ref{random-tree}.

\begin{theorem}[Restatement of Theorem~\ref{main}]
Let $k \geq 4$ be an integer. Let $\HH$ be a $2$-contractible $k$-tree with cross-cut number $\sigma$. 
The number of $\HH$-free $k$-graphs on $[n]$ is at most $2^{(\cross - 1 + o(1))\binom{n}{ k- 1}}.$
\end{theorem}
\begin{proof}
Since $\HH$ is $2$-contractible, it is $(k - 2)$-overlapping. 
Fix some $\ve > 0$, and let $\C$ be the collection returned by Theorem~\ref{counting} with $\ell=k-2$. Then every $\HH$-free  $k$-graph
$I$ on $[n]$ is a subgraph of some $C$ in $\C$. The number of such $I$ is at most $|\C| 2^{(\cross - 1 + \ve)\binom{n}{ k- 1}}$. 

Taking $n$ sufficiently large, we have that $|\C| \leq 2^{\ve \binom{n}{k - 1}}$, and thus the number of 
$\HH$-free $k$-graphs on $[n]$ is at most $2^{(\cross - 1 + 2 \ve)\binom{n}{k - 1}}$. As this holds for all $\ve$, the result follows. 
\end{proof}

\begin{theorem}[Restatement of Theorem~\ref{random-tree}]
Let $k \geq 4$ be an integer. Let $\HH$ be a $2$-contractible $\ell$-overlapping $k$-tree with cross-cut number $\sigma$. Then if $p \gg \frac{\log ^2 (n)}{n^{ k -\ell - 1}}$, with high probability 
$\ex( G(n, p), \Hc) = (1 - o(1)) (\sigma - 1 + o(1)) \binom{n}{ k - 1}$.
\end{theorem}
\begin{proof}
Fix some $\ve > 0$, and let $\C$ be the collection returned by Theorem~\ref{counting}. Then, let $X$ be the number of members $C \in \C$ such that $|C \cap G^{(k)}(n, p)| \geq (\cross - 1 + 2 \ve) p \binom{n}{k - 1}$. 

The probability that a specific member of $C \in \C$ satisfies $|C \cap G^{(k)}(n, p)| \geq (\cross - 1 + 2 \ve) p \binom{n}{k - 1}$ is at most $\exp\left( - \frac{\ve^2}{6 \sigma} p \binom{n}{k - 1}\right)$ by Chernoff's Inequality. 
Thus, the expected size $\mathbb{E}(X)$ of $X$ is at most 
$$\exp\left(\frac{K \log^2(n)}{n^{k - \ell - 1}} \binom{n}{k - 1}\right)  \exp\left( - \frac{\ve^2}{6 \sigma} p \binom{n}{k - 1}\right), $$

which goes to zero as $n \rightarrow \infty$ since $p \gg \frac{\log ^2 (n)}{n^{ k -\ell - 1}}$. Applying Markov's inequality, this implies that with high probability, $X = 0$. Hence,  with high probability, the largest $\T$-free subgraph of $G^{(k)}(n, p)$ is of size at most $(\cross - 1 +2 \ve)p \binom{n}{k - 1}$. As this holds for all $\ve$, the upper bound follows. 

A simple lower bound follows from the following: let $\Lc$ be the $k$-graph on $[n]$ vertices with the edge set $\{ \{i_1, i_2, \dots, i_k\} : i_i \in [\sigma - 1], \{i_2, \dots i_{k} \in \binom{[n] \setminus [\sigma]}{k - 1}\} $. 
Observe $\Lc$ contains no copy of $\Hc$ and furthermore has $(\sigma - 1) \binom{n - \sigma + 1}{k - 1}$ edges. 
Note that as $|G^{k}(n, p) \cap \Lc| \geq (\sigma - 1 - \ve)p \binom{n}{k -1}$ with probability $1 - \exp\left( - \frac{\ve^2}{6 \sigma} p \binom{n}{k - 1}\right)$ by Chernoff's Inequality, the lower bound follows.  
\end{proof}




\begin{thebibliography}{99}

\small


\bibitem{BDDLS} J. Balogh, S. Das, M. Delcourt, H. Liu, M. Sharifzadeh: Intersecting families of discrete 
structures are typically trivial, \emph{J. Combin. Th. Ser. A} \textbf{132} (2015), 224-245.

\bibitem{BMS} J. Balogh, R. Morris, W. Samotij: Independent sets in hypergraphs, \emph{J. Amer. Math. Soc.}
\textbf{28} (2015), 669-709.


\bibitem{BNS} J. Balogh, B. Narayanan and J. Skokan: The number of hypergraphs without linear cycles, \emph{J. 
Combin. Th. Ser B} \textbf{134} (2019), 309–321.

\bibitem{BS-sym} J. Balogh and W. Samotij: The number of $K_{m,m}$-free graphs, \emph{Combinatorica} \textbf{31} (2011), 131–150.

\bibitem{BS-asym} J. Balogh and W. Samotij: The number of $K_{s,t}$-free graphs, \emph{J. Lond. Math. Soc.} \textbf{83} (2011), 368-388.



\bibitem{CG} D. Conlon, T. Gowers: Combinatorial theorems in sparse random sets, 
\emph{Ann. of Math.} \textbf{184} (2016), 367-454.

\bibitem{DEF} M. Deza, P. Erd\H{o}s, P. Frankl: Intersection properties of systems of finite sets,
\emph{Proc. Lond. Math. Soc. (3)} \textbf{36} (1978), 369-384.

\bibitem{Erdos-complete-partite} P. Erd\H{o}s: On extremal problems of
graphs and generalized graphs, \emph{Israel Journal of Mathematics}
\textbf{2} (1964), 183-190.




\bibitem{EK} P. Erd\H{o}s, D. J. Kleitman: On coloring graphs to maximize the portion of multicolored $k$-edges,
\emph{J. Combin. Th,} \textbf{5} (1968), 164-169.




\bibitem{EFR} P. Erd\H{o}s, P. Frankl, and V. R\"odl: The asymptotic number of graphs not containing a fixed subgraph and a problem for hypergraphs having no exponent, \emph{Graphs Combin.} \textbf{2} (1986), 113–121.

\bibitem{EKR-free} P. Erd\H{o}s, D. J. Kleitman, and B. L. Rothschild: Asymptotic enumeration of $K_n$-free graphs, Colloquio Internazionale sulle Teorie Combinatorie (Rome, 1973), Accad. Naz. Lincei, Rome, 1976, pp. 19–27.



\bibitem{ES-cube} P. Erd\H{o}s and M. Simonovits: Cube-supersaturated graphs and
related problems, \emph{Progress in Graph theory (Waterloo, Ont., 1982)},
203-218, Academic Press, Toronto, 1984.

\bibitem{ES-sat} P. Erd\H{o}s and M. Simonovits: Supersaturated graphs and hypergraphs, \emph{Combinatorica} \textbf{3} (1983), 181–192.

\bibitem{FMS} A. Ferber, G. McKinley, W. Samotij: Supersaturated sparse graphs and hypergraphs,
\emph{Int. Math. Res. Not.} \textbf{2020} No. 2, 378-402.


\bibitem{frankl-1977}
P. Frankl: On families of finite sets no two of which intersect in a singleton,
\emph{Bull. Austral. Math. Soc.}  {\bf 17} (1977),  125--134.


\bibitem{FF-1985} P. Frankl, Z. F\"uredi: Forbidding just one intersection, \emph{J. Combin. Th. Ser. A} \textbf{39}
(1985), 160-176.

\bibitem{FF-exact} P. Frankl,  Z. F\"uredi: Exact solution of some Tur\'an-type problems,
\emph{J. Combin. Th. Ser. A} \textbf{45} (1987), 226--262.


\bibitem{Furedi-1983} Z. F\"uredi:  On finite set-systems whose every intersection is a kernel of a star,
{\em Discrete Math.} \textbf{47} (1983), 129--132.

\bibitem{Furedi-C4} Z. F\"uredi: On the number of edges of quadrilateral-free graphs, \emph{J. Combin. Theory Ser. B} \textbf{68} (1996), 1–6.



\bibitem{Furedi-trees} Z. F\"uredi:
Linear trees in uniform hypergraphs, \emph{European J. Combinatorics},
\textbf{35} (2014), 264--272.


\bibitem{FJ-cycles} Z. F\"uredi, T. Jiang: Hypergraph Tur\'an numbers of linear cycles,
\emph{J. Combin. Th. Ser. A} \textbf{123} (2014), 252--270.

\bibitem{FJ-trees} Z. F\"uredi, T. Jiang: Tur\'an numbers of hypergraph trees,
arXiv:1505.03210.



\bibitem{FJS} Z. F\"uredi, T. Jiang, R. Seiver:
Exact solution of the hypergraph Tur\'an problem for $k$-uniform linear paths,
\emph{Combinatorica}, \textbf{34} (2014), 299--322.


\bibitem{Furedi-Ozkahya} Z. F\"uredi, L. \"Ozkahya: Unavoidable subhypergraphs: $a$-clusters,
\emph{J. Combin. Th. Ser. A} \textbf{118}  (2011), 2246--2256.

\bibitem{HKL} P.E. Haxell, Y. Kohayakawa, T. Łuczak: Turán’s extremal problem in random graphs: forbidding even cycles, \emph{J. Combin. Theory Ser. B} \textbf{64} (1995) 273–287.

\bibitem{JL} T. Jiang, S. Longbrake: Balanced supersaturation and Tur\'an numbers in random graphs,
\emph{Advances in Combinatorics}, 2024:3, 26 pp.

\bibitem{KW} D. J. Kleitman and K. J. Winston: On the number of graphs without 4-cycles, Discrete Math. \textbf{41} (1982), 167–172.

\bibitem{KKS} Y. Kohayakawa, B. Kreuter, A. Steger: An extremal problem for random graphs and the number of graphs with large even-girth, \emph{Combinatorica} \textbf{18} (1998) 101–120.


\bibitem{KMV-path-cycle} A. Kostochka, D. Mubayi, J. Verstra\"ete:  
Tur\'an problems and shadows I: Paths and Cycles, \emph{J. Combin. Th. Ser. A},
\textbf{129} (2015), 57--79.

\bibitem{KMV-trees} A. Kostochka, D. Mubayi, J. Versta\"ete: 
Tur\'an problems and shadows II: trees, \emph{J. Combin. Th. Ser. B},
\textbf{122} (2017), 457-478.


\bibitem{L79}
{L. Lov\'asz}: \emph{Combinatorial Problems and Exercises}, Problem 13.31. 
Akad\'emiai Kiad\'o, Budapest and North Holland, Amsterdam, 1979.

\bibitem{MS} R. Morris, D. Saxton: The number of $C_{2\ell}$-free graphs, \emph{Advances in Math.} \textbf{208}
(2016), 534-580.


\bibitem{MV-survey} D. Mubayi, J. Verstra\"ete: A survey of Tur\'an problems for expansions,
Recent trends in Combinatorics, 117-143, IMA Vol. Math. Appl., 159, 
Springer, New York, 2016.

\bibitem{MW} D. Mubayi, L. Wang: The number of triple systems without even cycles,
\emph{Combinatorica} \textbf{39} (2019), 679-704.

\bibitem{MY} D. Mubayi, L. Yepremyan: Random Tur\'an theorem for hypergraph cycles, 
arXiv:2007.10320v1.

\bibitem{theta} G. McKinley and S. Spiro: The random Tur\'an problem for theta graphs. arXiv:2305.16550.

\bibitem{NRS} B. Nagle, V. R\"odl, and M. Schacht: Extremal hypergraph problems and the regularity method, Topics in discrete mathematics, Algorithms Combin., vol. 26, Springer, Berlin, 2006, pp. 247–278.

\bibitem{Nie-expand} J. Nie: Random Tur\'an theorem for expansions of spanning subgraphs of tight trees. arXiv preprint arXiv:2305.04193, 2023.

\bibitem{Nie-cycle} J. Nie: Tur\'an theorems for even cycles in random hypergraph. \emph{J. of Combin. Th. Ser. B}, \textbf{167} pp. 23–54, 2024.

\bibitem{NS} J. Nie and S. Spiro: Random Tur\'an problems for hypergraph expansions. arXiv:2408.03406.



\bibitem{RS} V. R\"odl, M. Schacht: Extremal results in random graphs,  Erd\H{o}s Centennial, in:
\emph{Bolyai Soc. Math. Stud.}, \textbf{25}, 2013, pp. 535-583.



\bibitem{ST} D. Saxton, A. Thomason: Hypergraphs containers, \emph{Invent. Math. }\textbf{201} (2015), 1-68.

\bibitem{Schacht} M. Schacht: Extremal results for random discrete structures, \emph{Ann. of Math.} \textbf{184} (2016), 331-363.



\end{thebibliography}
\end{document}